\documentclass[EJP,preprint]{ejpecp}

\usepackage[T1]{fontenc}
\usepackage[utf8]{inputenc}
\usepackage{booktabs}
\usepackage{enumitem}
\usepackage{xcolor}

\hypersetup{
  colorlinks=true,
  citecolor=blue!55!black,
  linkcolor=blue!55!black,
  urlcolor=blue!55!black,
  pdftitle={From Ehrenfest to Heisenberg scales: a hierarchy of mixing times for a phase-randomized quantum baker walk},
  pdfauthor={Amir Sepehri},
  pdfsubject={Structured random walks on unitary groups and Haar mixing},
  pdfkeywords={random walk on a compact group, unitary design, Wasserstein mixing, quantum baker map, Haar measure}
}

\setlist{nosep}
\numberwithin{equation}{section}

\SHORTTITLE{Mixing times for a phase-randomized quantum baker walk}

\TITLE{From Ehrenfest to Heisenberg scales: a hierarchy of mixing times for a
phase-randomized quantum baker walk}

\AUTHORS{Amir~Sepehri\footnote{Independent researcher, San Diego, CA.
\EMAIL{sepehri@alumni.stanford.edu}}}


\AMSSUBJ{60B15}
\AMSSUBJSECONDARY{60J10 ; 81Q50 ; 81P45}

\SUBMITTED{September 24, 2026}
\ACCEPTED{}

\VOLUME{0}
\YEAR{2026}
\PAPERNUM{0}
\DOI{10.1214/YY-TN}

\ABSTRACT{The quantum baker map is a standard model of quantum chaos, but
rigorous analysis of how it spreads random perturbations remains difficult.
Following the repeated ``baker plus refreshed diagonal perturbation''
architecture of Schack and Caves, we study a tractable strong-noise model
based on the Balazs--Voros quantization: at each time step, its unitary
propagator is followed by a diagonal unitary whose entries are independent and
uniform on the unit circle.  The resulting random propagator is a structured
random walk on $\U(N)$, $N=2^k$.  We quantify how quickly the baker dynamics
spreads this basis-local randomness at three levels: ensemble means, two-copy
statistics, and the full law of the accumulated propagator.  The walk becomes
an exact unitary $1$-design at time $k+1$, while its fixed-accuracy $2$-design
mixing time is $\Theta(\log N)$.  Thus the induced state ensembles reproduce
Haar means and two-copy statistics on the Ehrenfest scale.  The full law mixes
much more slowly.  A phase-adapted path coupling gives normalized Wasserstein
mixing in $O_\varepsilon(N)$ steps, whereas metric-entropy and Haar small-ball
estimates give an $\Omega_\varepsilon(N/\log N)$ lower bound and asymptotically
maximal distance at every $t=o(N/\log N)$.  The law remains singular with
respect to Haar measure for every $t<N$.  The moment results and the
Wasserstein upper bound are driven by the unistochastic Markov kernel
associated with $B_N$, an affine dyadic chain whose nonconstant Fourier modes
vanish exactly after $k$ steps.  The resulting hierarchy separates
Ehrenfest-scale randomization of state ensembles from near-Heisenberg-scale
exploration of the full unitary group.}

\newcommand{\T}{\mathbb T}
\newcommand{\E}{\mathbb E}
\newcommand{\Pp}{\mathbb P}
\newcommand{\C}{\mathbb C}
\newcommand{\Z}{\mathbb Z}
\newcommand{\U}{\mathrm U}
\newcommand{\HS}{\mathrm{HS}}

\newcommand{\Tr}{\operatorname{Tr}}
\newcommand{\diag}{\operatorname{diag}}
\newcommand{\Ad}{\operatorname{Ad}}
\newcommand{\Sym}{\operatorname{Sym}}

\newcommand{\diam}{\operatorname{diam}}
\newcommand{\cH}{\mathcal H}
\newcommand{\cG}{\mathcal G}
\newcommand{\cD}{\mathcal D}
\newcommand{\cF}{\mathcal F}
\newcommand{\cM}{\mathcal M}

\newcommand{\one}{\mathbf 1}

\begin{document}

\section{Introduction}

The classical baker map is a canonical model of chaotic dynamics.  It
stretches phase space in one direction, compresses it in the other, and folds
the result back onto itself.  In binary coordinates, this action becomes a
shift, making exponential instability and mixing unusually transparent.  The
Balazs--Voros quantum baker map replaces this classical evolution by a
finite-dimensional unitary operator and has become a standard testbed for
asking how classical chaos manifests itself in quantum dynamics
\cite{BalazsVoros1989,Saraceno1990}.

Much of the classical transparency is lost after quantization.  Quantum
evolution is unitary rather than dissipative, and interference replaces the
classical transport of probability densities.  Moreover, semiclassical
correspondence controls the dynamics only until evolving structures reach
quantum resolution, at the Ehrenfest time.  Even for the baker map, rigorous
results on quantum variance, eigenstates, and spectral projections require
substantial semiclassical analysis
\cite{DegliEspostiNonnenmacherWinn2006,AnantharamanNonnenmacher2007,Shou2025}.
An alternative way to probe the dynamics is to inject controlled randomness
and ask how the quantum baker propagates it.

Schack and Caves took this approach in their study of hypersensitivity to
perturbations \cite{SchackCaves1993}.  At each iteration, they apply the baker
propagator and then choose a fresh diagonal perturbation.  The phases in their
model belong to a finite, spatially correlated family determined by
perturbation cells; the resulting state ensemble and the information needed
to describe it are the main objects of study.  For selected initial states and
finite dimensions, their numerical experiments produced position-basis
entropies and pairwise-angle statistics close to those of uniformly random
state vectors.  They also found that the information needed to track the
perturbation history can greatly exceed the entropy produced by averaging over
that history.

Subsequent work developed this information-theoretic picture.  Schack and
Caves extended the state-ensemble diagnostics to several quantum maps
\cite{SchackCaves1996}.  Scott, Brun, Caves, and Schack later compared
hypersensitivity, entropy growth, and fidelity decay across a family of quantum
bakers; short-time analysis and later-time numerics indicated that
hypersensitivity distinguishes the nontrivial bakers from an extremal
shift-like quantization more effectively than the other two diagnostics
\cite{ScottBrunCavesSchack2006}.  Rigorous results are available for related
coarse-grained dynamics: Soklakov and Schack proved asymptotically linear
entropy production, at the classical Kolmogorov--Sinai rate, in a
decoherent-histories model of the quantum baker \cite{SoklakovSchack2002}.

Taken together, these works show that baker dynamics can amplify weak,
structured perturbations into widely dispersed ensembles of states, make the
information required to resolve the perturbation history far exceed the
entropy produced by ignoring it, and, under suitable coarse-graining,
reproduce the classical Kolmogorov--Sinai rate of entropy production.  They
thereby support hypersensitivity to perturbations as an information-theoretic
signature of quantum chaos.  Each perturbation history determines an
accumulated random propagator, and the successive products form a random walk
on $\U(N)$.  To our knowledge, however, this walk has been studied only through
the ensembles of states obtained from selected initial conditions, not through
the distribution of the propagator itself on the unitary group.  In particular,
the existing work does not provide dimension-dependent convergence estimates
for either its low-order unitary moments or its full probability law.

We retain the same repeated ``baker plus refreshed diagonal perturbation''
architecture, but use stronger noise: at every iteration, each position basis
state receives an independent phase, uniform on the circle.  Our model is
therefore not a weak-noise approximation to the deterministic baker.  It is a
maximally phase-randomized and exactly tractable endpoint of the same
architecture.  This stronger noise makes sharp results possible across
several levels of observation; it is also the main limitation on how directly
those results apply to weakly perturbed or deterministic quantum systems.

Against this background, we ask how long the baker dynamics takes to spread
basis-local randomness until the accumulated evolution resembles a globally
random unitary.  A diagonal kick changes phases but leaves position
probabilities unchanged.  The baker step spreads those phases among position
states, the next kick refreshes them, and repeated alternation progressively
randomizes the propagator.

We study this randomization at three levels.  The first two retain the
state-ensemble viewpoint: for every initial density matrix, one-copy mixing
controls the ensemble-averaged state, while two-copy mixing controls
fluctuations, purities, and other quadratic statistics of the ensemble.  Our
estimates hold uniformly over the input density matrix.  At the strongest level, we
go beyond state ensembles and study the full distribution of the accumulated
propagator on $\U(N)$.  Thus the paper both quantifies state-ensemble
randomization and addresses the additional propagator-level question.  How
soon the evolution appears random depends on which of these levels an observer
can access: the full law retains substantially more information than its
first two moments.

We now formulate this evolution.  Write
$\U(N)=\{U\in M_N(\C):U^*U=I\}$ for the group of $N\times N$ unitary
matrices, where $M_N(\C)$ is the space of $N\times N$ complex matrices and
$A^*$ denotes the conjugate transpose of $A$.  Let $N=2^k$, $k\ge1$, let $F_N$ be the
periodic discrete Fourier transform, and set
\[
 B_N=F_N^*\begin{pmatrix}F_{N/2}&0\\0&F_{N/2}\end{pmatrix}
\]
for the Balazs--Voros one-step quantum propagator.  Starting from an arbitrary
deterministic $X_0\in\U(N)$, define
\begin{equation}\label{eq:walk-intro}
 X_{t+1}=D_{\xi_{t+1}}B_NX_t,
 \qquad
 D_\xi=\diag(e^{i\xi_0},\ldots,e^{i\xi_{N-1}}),
\end{equation}
where all scalar phases are independent and uniform on
$\T=\mathbb R/(2\pi\mathbb Z)$.  Thus $B_N$ transports amplitudes according
to the baker dynamics, $D_{\xi_t}$ injects fresh basis-local phase noise, and
$X_t$ is the accumulated noisy propagator.  Haar measure---the unique
probability measure on $\U(N)$ invariant under left and right multiplication
by fixed unitaries---is stationary, so it supplies the canonical benchmark
for completely unstructured unitary randomness.  Yet one step injects only
$N$ real random parameters into the
$N^2$-dimensional group.  Strong noise therefore does not make the complete
propagator immediately Haar-random.

Phase averaging exposes a classical Markov chain on the position states.  Its
transition kernel is
\[
 P_N(a,b)=|(B_N)_{ba}|^2.
\]
Since $B_N$ is unitary, both the row sums and the column sums of $P_N$ are
one, so $P_N$ is bistochastic.  More specifically, it is unistochastic: its
entries are the squared moduli of the entries of a unitary matrix.  After the
phases have been averaged out, $P_N$ describes how the expected position
populations are transported by one application of $B_N$.  It is a noisy
doubling map on $\Z_N$: its Fourier multiplier is a
discrete tent function, and dyadic stretching sends every nonconstant
frequency to a zero of that multiplier in at most $k$ steps.  Consequently
$P_N^k=\Pi$ exactly, with $\Pi$ the uniform kernel.  This chain is the common
mechanism behind our results.  It gives exact first-moment mixing,
organizes the two-copy calculation, and controls the collision process in the
Wasserstein coupling.

The three levels are formalized by comparison with Haar randomness.  At the
first level, we ask when averaging $X_tAX_t^*$ gives the Haar answer for every
matrix $A$.  This is the unitary $1$-design question and tests mean states and
observables.  At the second, we ask the analogous question on two tensor
copies, measured in diamond norm.  This $2$-design question tests a broad class
of fluctuations and quadratic statistics.  At the third, we compare the
complete law of $X_t$ with Haar measure in Wasserstein and total variation
distance.  Although these notions come from different communities, they all
measure how far the baker dynamics has propagated the injected phase
randomness, but at different levels of resolution.  The one- and two-copy
comparisons retain only the corresponding low-order moments, whereas
Wasserstein and total variation compare the full probability laws.

The resulting hierarchy separates two physical time scales.  One-copy
information is erased exactly after $k+1$ steps, and two-copy information is
erased to fixed accuracy in $\Theta(\log N)$ steps.  Both occur on the
Ehrenfest scale: the classical baker doubles spatial frequencies at every
iteration, so frequencies of order one reach the quantum scale $N$ after
$\log_2 N$ steps.  By contrast, the complete propagator law remains nearly
maximally far from Haar for $t=o(N/\log N)$ and cannot even be absolutely
continuous before $N$ steps.  Its mixing window is therefore near the
Heisenberg scale, which is set by the inverse mean eigenphase spacing and is
of order $N$.  In this sense the
baker spreads enough of the injected noise to randomize low-order experiments
long before it has spread enough to randomize the full propagator law.

The questions studied here belong to two related literatures.  Probability on
compact groups studies convergence of convolution powers and quantitative transport
to Haar measure \cite{Varju2013,Borda2021}.  Quantum-information theory uses
unitary designs to retain finitely many Haar moments without requiring the
complete distribution to be Haar
\cite{DankertEtAl2009,HarrowLow2009,BrandaoHarrowHorodecki2016}.  Alternating
diagonal randomness in complementary bases is known to produce designs
efficiently \cite{NakataEtAl2017TwoDesigns,NakataEtAl2017Pseudorandomness};
the baker propagator studied here is different because its associated
unistochastic kernel is highly nonuniform and has exact zeros.  The use of
squared unitary matrix elements as a classical transition kernel has
precedents in quantum chaos.  Tanner related the spectral statistics of
unitary ensembles with prescribed squared moduli to the mixing properties of
the associated Markov chains \cite{Tanner2001}, while Berkolaiko, Keating, and
Smilansky used such kernels to construct quantum graphs from interval maps and
prove quantum ergodicity \cite{BerkolaikoKeatingSmilansky2007}.  Those works
connect a fixed unitary, or a static ensemble of unitaries, to an underlying
classical chain through spectral and eigenstate questions.  Here the role of
the chain is dynamical: after phase averaging, it governs the successive
moments of an accumulated random product.

For the baker map, the connection is especially transparent.  In
deterministic semiclassics, dyadic frequency doubling generates progressively
finer scales until it reaches quantum resolution, marking the usual limit of
Egorov correspondence.  In the present stochastic model, exact phase
averaging converts the same frequency cascade into finite-time first-moment
equilibration.

\subsection{Main results}

Write $H_N$ for Haar probability on $\U(N)$ and $\nu_{N,t}^{X_0}$ for the law
of $X_t$.  We use the normalized Frobenius metric
\[
 \rho_N(U,V)=N^{-1/2}\|U-V\|_{\HS}.
\]
Here
$\|A\|_{\HS}=(\Tr A^*A)^{1/2}$ is the Hilbert--Schmidt, or Frobenius, norm.
This normalization is natural in a sequence of growing unitary groups.  It
keeps typical Haar distances of order one and has the interpretation
\[
 \rho_N(U,V)^2=\frac1N\sum_{j=1}^N\|(U-V)e_j\|^2,
\]
the mean squared discrepancy over an orthonormal family of input states.
Nothing in the proofs depends essentially on this choice.  For the ordinary
Frobenius metric $d_F=\|\cdot\|_{\HS}=\sqrt N\,\rho_N$, however, fixed
accuracy is finer by a factor $N^{-1/2}$ in normalized units.  Thus
\eqref{eq:full-upper} gives the corresponding fixed-accuracy upper bound
$O_\varepsilon(N\log N)$ rather than $O_\varepsilon(N)$; the extra logarithm
is the number of constant-factor contraction blocks needed to reduce an
order-$\sqrt N$ initial distance to order one.
For a metric $d$, $W_{p,d}$ denotes the order-$p$ Wasserstein distance,
\[
 W_{p,d}(\mu,\nu)
 =\inf_{(X,Y)}\bigl(\E d(X,Y)^p\bigr)^{1/p},
\]
where the infimum is over all couplings with $X\sim\mu$ and $Y\sim\nu$.
We use the convention
$\|\mu-\nu\|_{\mathrm{TV}}=\sup_A|\mu(A)-\nu(A)|$ for total variation
distance.
The first theorem concerns the complete probability law.

\begin{theorem}[Full-law Wasserstein mixing]\label{thm:full-upper}
For every dyadic $N$, every $X_0\in\U(N)$, and every $t\ge0$,
\begin{equation}\label{eq:full-upper}
 W_{2,\rho_N}(\nu_{N,t}^{X_0},H_N)
 \le \pi\,2^{-\frac12\lfloor t/(4N+1)\rfloor}.
\end{equation}
Consequently, for every fixed $\varepsilon>0$, the normalized Wasserstein
mixing time is $O_\varepsilon(N)$.
\end{theorem}

At fixed normalized accuracy, the upper bound is optimal up to a logarithmic
factor.

\begin{theorem}[Full-law obstructions]\label{thm:full-lower}
Uniformly over deterministic starting points,
\begin{equation}\label{eq:maximal-lower}
 t=o(N/\log N)
 \quad\Longrightarrow\quad
 W_{1,\rho_N}(\nu_{N,t}^{X_0},H_N),
 W_{2,\rho_N}(\nu_{N,t}^{X_0},H_N)\longrightarrow\sqrt2.
\end{equation}
Moreover,
\begin{equation}\label{eq:tv-singular}
 \|\nu_{N,t}^{X_0}-H_N\|_{\mathrm{TV}}=1,
 \qquad 0\le t<N.
\end{equation}
For every fixed $\varepsilon>0$, there are $c_\varepsilon>0$ and
$N_\varepsilon<\infty$ such that, in the unnormalized Frobenius metric
$d_F(U,V)=\|U-V\|_{\HS}$,
\begin{equation}\label{eq:absolute-lower-theorem}
 t\le c_\varepsilon N,\quad N\ge N_\varepsilon
 \quad\Longrightarrow\quad
 W_{1,d_F}(\nu_{N,t}^{X_0},H_N),
 W_{2,d_F}(\nu_{N,t}^{X_0},H_N)\ge\varepsilon.
\end{equation}
\end{theorem}

The limit $\sqrt2$ in \eqref{eq:maximal-lower} is extremal: independent
coupling with Haar gives $W_{2,\rho_N}(\mu,H_N)\le\sqrt2$ for every probability
measure $\mu$ on $\U(N)$, since the expected squared cost of that coupling is
$2-2N^{-1}\operatorname{Re}\Tr(U^*\E V)=2$.

We next turn to low-order probes.  For a probability measure
$\mu$ on $\U(N)$ and $m\in\{1,2\}$, define its $m$-copy averaged
conjugation channel by
\begin{equation}\label{eq:twirl-def}
 \cG_\mu^{(m)}(A)=\int_{\U(N)}U^{\otimes m}A(U^*)^{\otimes m}\,d\mu(U).
\end{equation}
Here $U^{\otimes m}=U\otimes\cdots\otimes U$ is the $m$-fold tensor product,
acting on $(\C^N)^{\otimes m}$.
In quantum-information terminology this channel is the $m$-copy
\emph{twirl}; it averages the conjugation of an observable over the unitary
ensemble.  A measure is an exact unitary $m$-design when its twirl agrees
with the Haar-averaged channel, or Haar twirl.  Approximate designs will be
measured by the diamond norm of the difference.  Write $M_D(\C)$ for the
space of $D\times D$ complex matrices.  For a linear map
$\Phi:M_D(\C)\to M_D(\C)$, we use
\begin{equation}\label{eq:diamond-definition}
 \|\Phi\|_\diamond
 =\sup_{d\ge1}\|\Phi\otimes\mathrm{id}_d\|_{1\to1},
 \qquad
 \|\Psi\|_{1\to1}
 =\sup_{A\ne0}\frac{\|\Psi(A)\|_1}{\|A\|_1},
\end{equation}
where $\|A\|_1=\Tr\sqrt{A^*A}$ is the trace norm.  It is enough to take
$d=D$.  Here $\Phi\otimes\mathrm{id}_d$ is the tensor product of linear maps,
where $\mathrm{id}_d$ is the identity map on $M_d(\C)$; on elementary tensors,
\[
 (\Phi\otimes\mathrm{id}_d)(A\otimes B)=\Phi(A)\otimes B,
\]
and the action on all matrices is determined by linearity.  The auxiliary
identity map allows the input to be correlated with an untouched ancillary
system; taking $d=1$ recovers the ordinary induced trace norm.  More
generally, $\|\Psi\|_{p\to p}$ denotes the norm induced by the
Schatten $p$-norm: the trace norm for $p=1$, the Hilbert--Schmidt norm for
$p=2$, and the usual matrix operator norm for $p=\infty$.  For $m=1$,
equality with the Haar twirl says concretely that
\[
 \int UAU^*\,d\mu(U)=\frac{\Tr A}{N}I
\]
for every $A\in M_N(\C)$.

\begin{theorem}[Low-order moment mixing]\label{thm:design-main}
The law $\nu_{N,t}^{X_0}$ is an exact unitary $1$-design if and only if
$t\ge k+1$.  Moreover, with
\[
 c_*:=\prod_{j=1}^{\infty}(1-2^{-j})>0,
\]
one has, for $1\le t\le k$,
\begin{equation}\label{eq:design-lower-main}
 \|\cG_{\nu_{N,t}^{X_0}}^{(m)}-\cG_{H_N}^{(m)}\|_\diamond\ge c_* ,
 \qquad m=1,2.
\end{equation}
There are absolute constants $A,C>0$ and $N_0<\infty$ such that, for every
dyadic $N\ge N_0$ and $0<\varepsilon<1$,
\begin{equation}\label{eq:design-time}
 t\ge A\log N+C\log(\varepsilon^{-1})
 \quad\Longrightarrow\quad
 \|\cG_{\nu_{N,t}^{X_0}}^{(2)}-\cG_{H_N}^{(2)}\|_\diamond
 \le\varepsilon.
\end{equation}
The constants are independent of the deterministic starting point $X_0$.
\end{theorem}

More generally, for every fixed $0<\varepsilon<1$, neither the one-copy nor
the two-copy twirl can be within diamond distance $\varepsilon$ of Haar before
$k-O_\varepsilon(1)$ steps.  Thus the fixed-accuracy $2$-design time is
$\Theta_\varepsilon(\log N)$.  Here time counts applications of the complete
noisy baker propagator.  In one step, the global $N$-dimensional baker unitary
is applied and all $N$ diagonal phases are refreshed simultaneously.  The
$O(\log N)$ result therefore concerns the number of noisy dynamical steps; it
is not a circuit-depth bound for implementing the evolution with local gates
on $\log_2N$ qubits.

Combining the theorems gives the hierarchy
\begin{equation}\label{eq:hierarchy}
\begin{array}{c|c}
\text{notion of randomization}&\text{time scale}\\ \hline
\text{exact classical mixing}&\log_2N\\
\text{exact unitary $1$-design}&\log_2N+1\\
\text{$\varepsilon$-approximate unitary $2$-design}&\Theta_\varepsilon(\log N)\\
\text{normalized full-law Wasserstein mixing}
&\Omega_\varepsilon(N/\log N)\ \text{to}\ O_\varepsilon(N)\\
\text{total variation}&\text{maximal for }t<N.
\end{array}
\end{equation}
Since $N$ is inverse effective Planck scale, $\log N$ is the Ehrenfest scale
and $N$ the Heisenberg scale for the map.  More precisely, the classical baker
has Lyapunov exponent $\log2$ per iteration, so its Ehrenfest time is
$t_{\mathrm E}\sim(\log N)/(\log2)=k$, while the $N$-level quasienergy spacing
gives $t_{\mathrm H}\asymp N$.  The conclusion is not that the deterministic
baker is Haar-random, but that repeated fresh phase perturbations erase low
moments much earlier than they erase the global geometry of the propagator
law.

The hierarchy can also be read in terms of what an observer measures.  A
$1$-design gives the Haar answer for every averaged quantity of the form
$X_tAX_t^*$.  A $2$-design also gives the Haar answer, up to the stated error,
for experiments involving two copies of $X_t$; these include many variances,
purities, and other quadratic statistics.  Wasserstein and total variation
ask the much stronger question of whether the entire probability law is
Haar-like.  The walk can therefore look fully random to all one- and two-copy
experiments while its complete distribution remains far from Haar.

\subsection{Proof mechanisms and organization}

The upper bound in Theorem~\ref{thm:full-upper} comes from coupling two copies
of the walk that start at nearby unitary matrices.  Their small separation is
represented by a skew-Hermitian tangent matrix.  At each step we make a small
change to the fresh phases in the second copy.  This change cancels the
diagonal part of the tangent matrix, while the off-diagonal part remains as
the displacement that still has to be removed.

The dyadic structure of $B_N$ supplies the contraction mechanism.  Label an
off-diagonal entry by its two indices $(a,b)$.  After averaging over the fresh
phases, its squared magnitude is transported exactly as if $a$ and $b$ were
two independent $P_N$-chains.  When the two chains meet, the corresponding
entry has become diagonal and is cancelled by the phase adjustment.  The
exact $k$-step mixing of $P_N$ implies that the mean meeting time is less than
$2N$; Markov's inequality then shows that a block of $4N+1$ steps reduces the
expected squared displacement by at least one half.

Two points complete the coupling argument.  First, the phase adjustment at
step $s$ depends only on phases from earlier steps.  The resulting triangular
change of variables preserves the joint product-uniform phase law exactly,
so both coupled trajectories have the correct marginal distribution.
Second, the construction initially contracts only sufficiently close
starting points.  Oliveira's local-to-global Wasserstein theorem
\cite{Oliveira2009} extends the same contraction along paths in $\U(N)$ and
hence to arbitrary starting laws.

Theorem~\ref{thm:full-lower} is geometric.  After $t$ steps the law is supported
on the image of only $Nt$ phase parameters, while $\U(N)$ has dimension
$N^2$.  A quantitative covering estimate for that image, combined with Haar
small-ball bounds, yields \eqref{eq:maximal-lower}; the same dimension mismatch
gives \eqref{eq:tv-singular} directly.

The moment proofs begin with the same operation: average over a fresh set of
diagonal phases.  What survives that average depends on whether we observe
one copy or two.

For one copy, the phase average simply deletes every off-diagonal matrix
entry.  The remaining $N$ diagonal entries form a population vector.
Subsequent averaged baker steps move this vector according to the classical
kernel $P_N^{\mathsf T}$.  Since $P_N^k$ is exactly the uniform kernel, the
first phase average followed by $k$ population transitions gives the Haar
one-copy average.  This proves the upper bound $k+1$.

To see that $k+1$ is also necessary, we start with a particular oscillating
population profile---a Fourier wave around the $N$ basis states.  The matrix $B_N$
doubles its frequency at every step.  Its amplitude stays bounded away from
zero through time $k$; the next transition encounters a zero Fourier
multiplier and removes it.  Thus a one-copy experiment can still distinguish
the walk from Haar through time $k$.  The same obstruction also applies to two copies: if every
two-copy experiment had already reached its Haar value, then ignoring one
copy would give the Haar answer for every one-copy experiment.

For two copies, phase averaging retains more information.  A matrix term
carrying two basis labels survives whenever the same two labels occur on its
input and output sides, possibly in the opposite order.  Passing to symmetric
and antisymmetric combinations of the two orderings separates the surviving
information into four blocks.  Two are ordinary population distributions on
pairs of basis states.  The other two record coherence between the symmetric
and antisymmetric descriptions of a pair.

The coherence blocks disappear first.  After a harmless change of diagonal
phases makes $B_N$ real, their evolution is the exterior square of
the one-particle kernel $P_N^{\mathsf T}$.  Concretely, this exterior square
measures the two-dimensional information carried by a pair of population
vectors.  After $k$ steps, $P_N^k$ sends every population vector into the
single uniform direction.  A rank-one map carries no two-dimensional
information, so both coherence blocks are then exactly zero.

The two population blocks require more work because they are Markov chains
on order $N^2$ pairs.  Translate both labels of a pair by the same amount.
This changes the common location of the pair but not the distance between its
entries.  Dyadic frequency doubling removes the common-location information
after $k$ steps, leaving only two chains on the $O(N)$ possible separations.
We write those smaller chains explicitly in a cosine basis.  A weighted norm
of the cosine coefficients contracts by a fixed factor at each subsequent
step.  The resulting error has the form $CN^3\rho^s$ with $\rho<1$, where
the factor $N^3$ accounts for reconstructing the full two-copy operator and
passing to diamond norm.  Taking $s=O(\log N+\log\varepsilon^{-1})$ absorbs
that polynomial factor.  Together with the initial $k=\log_2N$ focusing
steps, this gives the stated $O(\log N)$ two-design upper bound.

Section~\ref{sec:classical} develops the finite chain used throughout the
paper.  Sections \ref{sec:full-upper} and \ref{sec:full-lower} prove the
full-law bounds by rather different arguments: coupling for the upper bound
and geometry of the reachable set for the lower bound.  Section
\ref{sec:one-design} settles the first moment and the moment lower bounds.
Sections \ref{sec:two-copy}--\ref{sec:weighted} analyze the second moment in
three stages: phase-sector decomposition, translation focusing, and weighted
Fourier contraction.  Section \ref{sec:design-proof} assembles those estimates.
The final section discusses the scope of the hierarchy and open sharp-time
questions.

\section{The unistochastic Markov chain}\label{sec:classical}

Throughout, $N=2^k$, $M=N/2$, and indices are elements of $\Z_N$.  Our Fourier
convention is
\[
 (F_N)_{xy}=N^{-1/2}e^{-2\pi ixy/N}.
\]

To see where the classical chain comes from, let $\psi$ be a state just
before a fresh diagonal phase and a baker step.  The amplitude at $b$ after
these operations is
\[
 \sum_a(B_N)_{ba}e^{i\xi_a}\psi_a.
\]
Independence and uniformity of the phases remove every cross term with
$a\ne a'$ when its squared modulus is averaged.  Hence
\[
 \E_\xi\left|\sum_a(B_N)_{ba}e^{i\xi_a}\psi_a\right|^2
 =\sum_a |(B_N)_{ba}|^2|\psi_a|^2.
\]
Thus the expected position populations are transported by the Markov kernel
\[
 P_N(a,b)=|(B_N)_{ba}|^2.
\]
We call $P_N$ the unistochastic Markov kernel associated with $B_N$.  It
describes the classical population dynamics obtained by averaging out the
phase information; it does not replace the entire unitary evolution by a
classical chain.  The same kernel will later describe both the motion of
diagonal matrix entries and the motion of off-diagonal tangent energy.

We establish three facts about $P_N$.  It is a noisy doubling map, so its
Fourier evolution is explicit.  It sends every starting point exactly to the
uniform distribution after $k=\log_2N$ steps.  Finally, two independent
copies collide in expected time less than $2N$.  The exact mixing statement
will be used in the moment arguments, while the collision estimate will be
used in the full-law coupling.

\subsection{The unistochastic kernel}

Write an input index as $a=a_0+hM$, where $0\le a_0<M$ and
$h\in\{0,1\}$.  Thus $a_0$ is the position within one half of the basis and
$h$ specifies which half contains $a$.  Direct multiplication of the Fourier
blocks gives
\begin{equation}\label{eq:baker-entry}
 (B_N)_{ba}=\frac{(-1)^{bh}}{\sqrt{NM}}
 \sum_{r=0}^{M-1}e^{2\pi ir(b-2a_0)/N}.
\end{equation}
Taking the absolute square removes the sign $(-1)^{bh}$.  Since
$2a=2a_0$ modulo $N$, the result depends on $a$ and $b$ only through
$b-2a$:
\begin{equation}\label{eq:affine-kernel}
 P_N(a,b)=q_N(b-2a),\qquad
 q_N(z)=\frac1{NM}\left|\sum_{r=0}^{M-1}e^{2\pi irz/N}\right|^2.
\end{equation}
The geometric sum yields the concrete probability mass function
\begin{equation}\label{eq:q-explicit}
 q_N(z)=
 \begin{cases}
 1/2,&z=0,\\
 0,&z\ne0\text{ even},\\
 \displaystyle\frac{2}{N^2\sin^2(\pi z/N)},&z\text{ odd}.
 \end{cases}
\end{equation}
In particular, $q_N$ puts mass $1/2$ at zero, vanishes at every other even
site, and distributes its remaining mass over the odd sites.  A $P_N$-chain
can therefore be generated by first doubling its current position and then
adding independent noise with law $q_N$:
\begin{equation}\label{eq:affine-chain}
 Y_{s+1}=2Y_s+Z_{s+1}\pmod N,
 \qquad Z_s\stackrel{\mathrm{iid}}\sim q_N.
\end{equation}
Here ``iid'' means that the variables $Z_s$ are independent and identically
distributed.

The next lemma gives the Fourier transform of this noise.  Its tent shape is
the reason dyadic doubling eventually destroys every nonconstant Fourier
mode.  For a function $f$ on $\Z_N$, our convention is
\[
 \widehat f(r)=\sum_{x\in\Z_N}f(x)e^{-2\pi irx/N}.
\]

\begin{lemma}[Tent multiplier]\label{lem:tent}
With $|r|_N=\min(r,N-r)$,
\begin{equation}\label{eq:qhat}
 \widehat q_N(r)=\sum_{z\in\Z_N}q_N(z)e^{-2\pi irz/N}
 =1-\frac{|r|_N}{M}.
\end{equation}
\end{lemma}

\begin{proof}
Expanding the square in the definition of $q_N$ and using character
orthogonality gives
\[
 \widehat q_N(r)
 =\frac1M\#\{(u,v)\in\{0,\ldots,M-1\}^2:u-v\equiv r\pmod N\}.
\]
The possible differences $u-v$ lie strictly between $-N$ and $N$.  For the
representative of $r$ having absolute value $|r|_N$, there are therefore
$M-|r|_N$ admissible pairs.  Dividing by $M$ proves the formula.
\end{proof}

\subsection{Exact finite-time mixing}

Let $p_t^{(N)}$ be the law of \eqref{eq:affine-chain} at time $t$ when
$Y_0=0$.  The affine update gives
$\widehat p_{t+1}^{(N)}(r)=\widehat q_N(r)\widehat
p_t^{(N)}(2r)$, and iteration gives
\begin{equation}\label{eq:pt-fourier}
 \widehat p_t^{(N)}(r)=\prod_{j=0}^{t-1}\widehat q_N(2^jr).
\end{equation}

\begin{proposition}[Exact dyadic mixing]\label{prop:exact-classical}
For every starting point $a\in\Z_N$,
\begin{equation}\label{eq:Pk-uniform}
 P_N^k(a,\cdot)=\operatorname{Unif}(\Z_N).
\end{equation}
Moreover, for $0\le t\le k$,
\begin{equation}\label{eq:pt-infty}
 \|p_t^{(N)}\|_\infty=p_t^{(N)}(0)=2^{-t}.
\end{equation}
\end{proposition}

Here $\|p\|_\infty=\max_x|p(x)|$; for a probability mass function, this is
its largest point mass.

\begin{proof}
We first prove exact mixing.  Fix a nonzero frequency $r$.  Because
$N=2^k$, repeatedly doubling $r$ modulo $N$ reaches $N/2$ before it reaches
zero.  Lemma~\ref{lem:tent} gives $\widehat q_N(N/2)=0$, so one of the first
$k$ factors in \eqref{eq:pt-fourier} vanishes.  Thus
$\widehat p_k^{(N)}(r)=0$ for every $r\ne0$, while the constant Fourier
coefficient remains one.  This characterizes the uniform distribution.
Starting from $a$ instead of zero only adds the deterministic term $2^ka$,
which is zero modulo $N$, and hence \eqref{eq:Pk-uniform} holds for every
starting point.

It remains to identify the largest point mass before exact mixing.  Every
factor in \eqref{eq:pt-fourier} is nonnegative.  Fourier inversion and the
triangle inequality therefore give $p_t^{(N)}(x)\le p_t^{(N)}(0)$ for every
$x$.  To evaluate the mass at zero, unwind the recursion:
\[
 Y_t=\sum_{j=1}^t2^{t-j}Z_j\pmod N.
\]
By \eqref{eq:q-explicit}, each $Z_j$ is either zero or odd.  If at least one
$Z_j$ is nonzero, let $j$ be the last such index.  Its contribution has
exactly $t-j$ factors of two, while every earlier contribution has more.
Consequently a nonzero value of $Y_t$ cannot be a multiple of $2^t$.  For
$t<k$, the point $M=N/2$ is a nonzero multiple of $2^t$, so
$p_t^{(N)}(M)=0$.

There is also a simple one-step recursion for the mass at zero.  Since $2x$
is even, \eqref{eq:q-explicit} shows that the transition from $x$ to zero has
positive probability only for $x=0$ or $x=M$, and in either case that
probability is $1/2$.  Hence, for $t<k$,
\[
 p_{t+1}^{(N)}(0)=\sum_xp_t^{(N)}(x)q_N(-2x)
 =\frac12\{p_t^{(N)}(0)+p_t^{(N)}(M)\}
 =\frac12p_t^{(N)}(0).
\]
Starting from $p_0^{(N)}(0)=1$ gives $p_t^{(N)}(0)=2^{-t}$ up to time
$k$.  At $t=k$ this also agrees with the uniform law already proved.
\end{proof}

\subsection{Two-particle collisions}

We next need to know how long two independent $P_N$-particles can avoid one
another.  Let $Y_s,Y_s'$ be independent copies of
\eqref{eq:affine-chain} and set $\Delta_s=Y_s-Y_s'$.  Then
\[
 \Delta_{s+1}=2\Delta_s+(Z_{s+1}-Z_{s+1}')\pmod N.
\]
The increment $Z_{s+1}-Z_{s+1}'$ has a symmetric distribution, and the
transition kernel $P_N^\Delta$ of the difference chain is doubly stochastic.
Moreover, after $k$ steps each particle is uniform independently of its
starting point.  Their difference is therefore uniform as well, so
$(P_N^\Delta)^k=\Pi$.  A collision of the two particles is exactly a visit
of the difference chain to zero.  Accordingly, define
\[
 \tau_0=\inf\{s\ge0:\Delta_s=0\}.
\]

\begin{proposition}[Collision bound]\label{prop:collision}
For every nonzero initial difference,
\begin{equation}\label{eq:collision-mean}
 \E_d\tau_0<2N,
 \qquad
 \Pp_d(\tau_0>4N)<\frac12.
\end{equation}
\end{proposition}

The subscript $d$ on $\E_d$ and $\Pp_d$ means that the difference chain
starts from $\Delta_0=d$.

\begin{proof}
Because the chain is exactly mixed after $k$ steps, its centered potential
kernel is the finite sum
\[
 Z=\sum_{s=0}^{k-1}\{(P_N^\Delta)^s-\Pi\}
\]
and satisfies $(I-P_N^\Delta)Z=I-\Pi$.  We briefly derive the hitting-time
identity that we need.  Define
\[
 h(d)=N\{Z(0,0)-Z(d,0)\}.
\]
Clearly $h(0)=0$.  For $d\ne0$, the $(d,0)$ entry of
$(I-P_N^\Delta)Z=I-\Pi$ is $-1/N$, and therefore
\[
 h(d)-\sum_eP_N^\Delta(d,e)h(e)=1.
\]
These are exactly the first-step equations for the expected time to hit
zero: the value is zero at the target, and from every other state one unit
of time elapses before the chain moves.  The unique solution of these
finite-state Dirichlet equations is $h(d)=\E_d\tau_0$.  Hence
\[
 \E_d\tau_0=N\{Z(0,0)-Z(d,0)\}.
\]
Thus the problem reduces to bounding the total return probability at zero.
Indeed, since $(P_N^\Delta)^s(d,0)\ge0$,
\[
 Z(0,0)-Z(d,0)
 \le\sum_{s=0}^{k-1}(P_N^\Delta)^s(0,0).
\]
When both particles start at zero, they agree at time $s$ with probability
\[
 (P_N^\Delta)^s(0,0)
 =\sum_xp_s^{(N)}(x)^2
 \le \|p_s^{(N)}\|_\infty.
\]
Proposition~\ref{prop:exact-classical} bounds the last quantity by $2^{-s}$.
The resulting geometric sum is smaller than $2$, proving
$\E_d\tau_0<2N$.  Finally,
$\Pp_d(\tau_0>4N)\le \E_d\tau_0/(4N)<1/2$ by Markov's inequality.
\end{proof}

Let $\cM_N$ be the substochastic transition matrix of the pair $(Y,Y')$ on
the off-diagonal state space
\[
 \{(a,b)\in\Z_N^2:a\ne b\},
\]
obtained by killing the pair when its two coordinates meet.  Substochastic
means that its entries are nonnegative and its row sums are at most one; the
missing mass is precisely the probability of having been killed.  Its powers have
the direct probabilistic interpretation
\[
 \cM_N^s\one(a,b)=\Pp_{a,b}(\tau_0>s).
\]
Here $\one$ without a subscript is the constant-one function on the
off-diagonal pair space.
Write $r(\cM_N)$ for the spectral radius of $\cM_N$, the largest absolute
value of one of its eigenvalues.  Bounding this number
away from one will later give exponential decay of the off-diagonal tangent
energy in the Wasserstein argument.

\begin{corollary}\label{cor:killed-gap}
The killed-pair kernel satisfies
\[
 r(\cM_N)\le1-\frac1{2N}.
\]
\end{corollary}

\begin{proof}
By the Perron--Frobenius theorem for nonnegative matrices, there is a
nonnegative eigenvector $f$ with eigenvalue $r(\cM_N)$.  Normalize it by
$\|f\|_\infty=1$.  Choose a distinct pair $x$ at which $f(x)=1$.  Since
$f\le\one$ and $\cM_N$ has nonnegative entries,
\[
 \cM_N^s\one(x)\ge \cM_N^sf(x)=r(\cM_N)^s.
\]
The left-hand side is the probability that the pair has not yet collided
after $s$ steps.  Summing over $s\ge0$ therefore gives
\[
 \frac1{1-r(\cM_N)}
 \le \E_x\tau_0<2N,
\]
where the last inequality is Proposition~\ref{prop:collision}.  Rearranging
proves the claim.
\end{proof}

\section{Full-law Wasserstein upper bound}\label{sec:full-upper}

Our goal is to couple two copies of the walk so that their distance contracts
by a fixed factor every $O(N)$ steps.  We first do this when the two starting
matrices are very close.  Oliveira's local-to-global path-coupling theorem
will then extend the estimate to arbitrary starting points.

We begin with the small amount of geometry needed to describe a nearby pair.
Write
\[
 \mathfrak u(N)=\{A\in M_N(\mathbb C):A^*=-A\}
\]
for the skew-Hermitian matrices.  If $A\in\mathfrak u(N)$, then $e^{hA}$ is
unitary, and $h\mapsto e^{hA}X$ is a path in $\U(N)$ starting at $X$.  Thus,
for small $h$, a pair of the form
\[
 Y=e^{hA}X
\]
may be read as follows: $A$ gives the direction from $X$ to $Y$, and $h$
gives its size when $\|A\|_{\HS}=1$.  This is the only role played below by
the Lie algebra $\mathfrak u(N)$.

We measure the length of such directions with the Hilbert--Schmidt inner
product
\[
 \langle A,C\rangle_{\HS}=\operatorname{Re}\Tr(A^*C),
 \qquad A^*=-A, C^*=-C.
\]
The distance $D(U,V)$ is the length of the shortest path in $\U(N)$ under
this norm.  Two elementary comparisons will be used:
\begin{equation}\label{eq:metric-compare}
 \|U-V\|_{\HS}\le D(U,V),\qquad \diam(\U(N),D)\le\pi\sqrt N.
\end{equation}
The first says that a path is at least as long as the straight chord joining
its endpoints.  For the second, diagonalize $UV^*$ and rotate each of its
$N$ eigenvalues to $1$ through an angle in $[-\pi,\pi]$; the resulting path
has length at most $\pi\sqrt N$.

As the coupled walks evolve, we describe their relative displacement by a
skew-Hermitian matrix $R$: the notation
$Y=e^{hR+O(h^2)}X$ means that, to first order in $h$, the second point is
obtained from the first by left multiplication by $e^{hR}$.  In differential
geometry this identification is called left trivialization of the tangent
bundle.  No coordinates on $\U(N)$ will be needed.  Conjugation preserves the
Hilbert--Schmidt norm, so transporting $R$ through a unitary operation does
not by itself change its length.

Here is the coupling argument in outline.  Start two copies of the walk a
small distance apart and follow the matrix $R$ describing their first-order
displacement.
At each step we shift the phases of the second copy so as to cancel the
diagonal part of that displacement.  The remaining off-diagonal squared
coefficients evolve like two $P_N$-particles until they collide.  The
collision estimate from Proposition~\ref{prop:collision} therefore contracts
the expected squared size of $R$ by a fixed factor in $O(N)$ steps.  The
local-to-global theorem then turns this first-order construction into a
Wasserstein contraction for arbitrary starting laws.

We next examine one layer of the coupling.  For any matrix $C$, define
\[
 \Pi_{\mathfrak d}C=\diag(C_{00},\ldots,C_{N-1,N-1}),
 \qquad QC=C-\Pi_{\mathfrak d}C.
\]
Thus $\Pi_{\mathfrak d}C$ is the diagonal part of $C$ and $QC$ is its
off-diagonal part.  We also use the abbreviation
$\Ad_U(C)=UCU^*$ for conjugation by a unitary $U$.

Suppose that, just before layer $s$, the two nearby trajectories satisfy
\[
 Y_{s-1}=e^{hR_{s-1}+O(h^2)}X_{s-1}.
\]
First apply $B_N$ to both trajectories.  The identity
$B_Ne^{hR}=e^{hB_NRB_N^*}B_N$, which holds because conjugation commutes with
the matrix exponential, shows that their new first-order displacement is
\[
 H_s:=B_NR_{s-1}B_N^*.
\]
Decompose it as
\[
 H_s=\Pi_{\mathfrak d}H_s+QH_s.
\]

Why can the first term be removed by the coupling?  Because a small change
of the next diagonal phase by $-hb_s$ multiplies the second trajectory by
$e^{-ih\diag(b_s)}$ and hence contributes
$-i\diag(b_s)$ to its first-order displacement.  Since $H_s$ is
skew-Hermitian, its diagonal entries are purely imaginary, so there is a
real vector $b_s$ satisfying
\[
 i\diag(b_s)=\Pi_{\mathfrak d}H_s.
\]
Using the phase correction $-hb_s$ therefore cancels the diagonal part of
$H_s$ and leaves $QH_s$.
Subsection~\ref{subsec:phase-coupling} will show that these corrections can
be made without changing the marginal distribution of either walk.

Finally, the common fresh phase $D_{\xi_s}$ changes the coordinates in which
the remaining displacement is expressed.  The residual displacement after
the layer is
\begin{equation}\label{eq:one-step-residual}
 R_s=D_{\xi_s}(QH_s)D_{\xi_s}^*
 =\Ad_{D_{\xi_s}}Q(B_NR_{s-1}B_N^*).
\end{equation}
Entry by entry,
\[
 (R_s)_{ij}=e^{i(\xi_{s,i}-\xi_{s,j})}(QH_s)_{ij}.
\]
The fresh phase therefore leaves every magnitude $|(QH_s)_{ij}|$ unchanged
but randomizes the relative phases of the off-diagonal entries.  When the
next baker matrix forms linear combinations of those entries, phase
averaging eliminates the cross terms.  The next subsection shows that the
resulting squared magnitudes evolve exactly as a classical two-particle
process.

\subsection{Phase averaging of the residual displacement}

The following identity is the bridge from contraction of the matrix $R$ to
the collision problem studied in Section~\ref{sec:classical}.

\begin{lemma}[Killed-pair evolution]\label{lem:phase-average}
If $C\in\mathfrak u(N)$ has zero diagonal and
$C^+=B_ND_\xi C D_\xi^*B_N^*$, then, for $i\ne j$,
\begin{equation}\label{eq:phase-average}
 \E_\xi|C^+_{ij}|^2
 =\sum_{a\ne b}P_N(a,i)P_N(b,j)|C_{ab}|^2.
\end{equation}
\end{lemma}

\begin{proof}
Expanding the $(i,j)$ entry gives
\[
 C^+_{ij}=\sum_{a\ne b}(B_N)_{ia}\overline{(B_N)_{jb}}
 e^{i(\xi_a-\xi_b)}C_{ab},
\]
where the restriction $a\ne b$ uses the zero diagonal of $C$.  When the
squared modulus is expanded, independence and uniformity of the phases make
the expectation of every cross term zero unless the two ordered pairs are
identical.  The surviving coefficient is
\[
 |(B_N)_{ia}|^2|(B_N)_{jb}|^2
 =P_N(a,i)P_N(b,j),
\]
which proves \eqref{eq:phase-average}.
\end{proof}

It is useful to spell out the probabilistic content of the identity.  Regard
$|C_{ab}|^2$ as energy carried by the ordered off-diagonal pair $(a,b)$.
After phase averaging, that energy moves to $(i,j)$ with weight
$P_N(a,i)P_N(b,j)$, exactly as if the two coordinates made independent
$P_N$-transitions.  Projection by $Q$ deletes the terms with $i=j$.  Thus
energy is removed precisely when the two particles meet, which explains the
killed chain introduced after Proposition~\ref{prop:collision}.

We now iterate this identity.  Start with a unit displacement
$R_0=A\in\mathfrak u(N)$ and define each subsequent $R_s$ by
\eqref{eq:one-step-residual}.
The order in Lemma~\ref{lem:phase-average} matches the transition from one
residual to the next: $R_s$ ends with conjugation by $D_{\xi_s}$, and the
next step begins with conjugation by $B_N$.

Here is the iteration explicitly.  Put
\[
 w_1(a,b)=|(Z_1)_{ab}|^2,
 \qquad Z_1=Q(B_NAB_N^*),
 \qquad a\ne b.
\]
If $w_j(a,b)$ denotes the expected off-diagonal energy at $(a,b)$ after
$j-1$ killed-pair transitions, Lemma~\ref{lem:phase-average} says
\[
 w_{j+1}(i,j')
 =\sum_{a\ne b}w_j(a,b)P_N(a,i)P_N(b,j'),
 \qquad i\ne j'.
\]
In other words, $w_{j+1}=w_j\cM_N$ in row-vector notation.  Induction gives
$w_{s+1}=w_1\cM_N^s$, and summing over all off-diagonal endpoints yields
\begin{equation}\label{eq:residual-survival}
 \E\|R_{s+1}\|_{\HS}^2
 =\sum_{a\ne b}|(Z_1)_{ab}|^2\Pp_{a,b}(\tau_0>s),
\end{equation}
For completeness, the indexing in \eqref{eq:residual-survival} can be read as
follows.  The first baker step produces the initial off-diagonal energy
$|(Z_1)_{ab}|^2$.  Each subsequent baker step is one transition of the
ordered pair $(a,b)$.  A contribution remains in $R_{s+1}$ exactly when the
pair has avoided the diagonal during those $s$ transitions.  Summing over
the surviving endpoint gives \eqref{eq:residual-survival}.  Notice also that
$\|Z_1\|_{\HS}\le\|A\|_{\HS}$ because $Q$ is an orthogonal projection.
Together with Proposition~\ref{prop:collision},
\begin{equation}\label{eq:residual-half}
 \E\|R_{4N+1}\|_{\HS}^2\le\frac12\|A\|_{\HS}^2.
\end{equation}

\subsection{A measure-preserving phase coupling}\label{subsec:phase-coupling}

We first isolate the elementary fact that keeps the modified phases random.
It is the probabilistic reason for making the correction at step $s$ depend
only on phases observed before step $s$.

\begin{lemma}[Triangular phase shifts]\label{lem:triangular-phase-shifts}
For $1\le s\le L$, let
$a_s:(\T^N)^{s-1}\to\mathbb R^N$ be smooth and periodic, with $a_1$
constant.  For every real $h$, the transformation
\begin{equation}\label{eq:abstract-triangular-shift}
 \xi_s'=\xi_s-ha_s(\xi_1,\ldots,\xi_{s-1})\pmod{2\pi},
 \qquad 1\le s\le L,
\end{equation}
is a measure-preserving bijection of $(\T^N)^L$.  In particular, if the
unprimed phases are jointly product-uniform, then so are the primed phases.
\end{lemma}

\begin{proof}
The inverse is recovered successively.  First recover
$\xi_1=\xi_1'+ha_1$; once $\xi_1,\ldots,\xi_{s-1}$ are known, recover
\[
 \xi_s=\xi_s'+ha_s(\xi_1,\ldots,\xi_{s-1})\pmod{2\pi}.
\]
Thus the map is a bijection.  To see directly that it preserves the product
measure, condition on $\xi_1,\ldots,\xi_{s-1}$.  The correction applied to
$\xi_s$ is then fixed, and translation by a fixed vector preserves the
uniform distribution on $\T^N$.  Applying this observation successively for
$s=1,\ldots,L$ shows that the transformed phase string is again
product-uniform.  Smoothness follows from the assumptions on the functions
$a_s$.
\end{proof}

We now construct the coupling.  Take $X$ and $Y$ sufficiently close that
their shortest relative logarithm is unambiguous.  We may then write
\[
 Y=e^{hA}X,\qquad \|A\|_{\HS}=1,\qquad h=D(X,Y).
\]
Let $X_s$ and $Y_s$ be the two trajectories.  The first uses the original
phases $\xi_s$; the second will use predictably shifted phases $\xi_s'$.
At step $s$, the transported displacement
$B_NR_{s-1}B_N^*$ is skew-Hermitian, so its diagonal entries are purely
imaginary.  There is therefore a real vector $b_s\in\mathbb R^N$ such that
\begin{equation}\label{eq:diagonal-correction}
 \Pi_{\mathfrak d}(B_NR_{s-1}B_N^*)=i\diag(b_s),
\end{equation}
and we give the second trajectory the corrected phase
\begin{equation}\label{eq:triangular-coupling}
 \xi_s'=\xi_s-hb_s(\xi_1,\ldots,\xi_{s-1})\pmod{2\pi}.
\end{equation}
The important point is that $b_s$ is determined by $R_{s-1}$ and hence by
the earlier phases; it is chosen before the new phase $\xi_s$ is revealed.
Lemma~\ref{lem:triangular-phase-shifts} therefore guarantees that the
shifted phases still have exactly the required product-uniform law.

Substituting the corrected phase into the one-step update verifies the
calculation given above.  If
$Y_{s-1}=e^{hR_{s-1}+O(h^2)}X_{s-1}$, then
\begin{align*}
Y_s
 &=D_{\xi_s-hb_s}B_NY_{s-1}\\
 &=e^{-ih\diag(b_s)}D_{\xi_s}
   e^{hB_NR_{s-1}B_N^*+O(h^2)}B_NX_{s-1}\\
 &=e^{h\Ad_{D_{\xi_s}}Q(B_NR_{s-1}B_N^*)+O(h^2)}X_s.
\end{align*}
This is exactly the recursion \eqref{eq:one-step-residual}.  For fixed $N$
and a fixed number of steps, the $O(h^2)$ error is uniform over all choices
of the phases.

For $L=4N+1$, each $b_s$ is a smooth periodic function of the preceding
phases.  Lemma~\ref{lem:triangular-phase-shifts} therefore shows that the
resulting triangular map, which we denote by
\[
 \Psi_h:(\T^N)^L\longrightarrow(\T^N)^L,
 \qquad (\xi_1,\ldots,\xi_L)\longmapsto
 (\xi_1',\ldots,\xi_L'),
\]
preserves product-uniform measure.

Let $K_N$ denote the one-step transition kernel of the walk.  We first verify
the two marginals of the coupling explicitly.  For an initial
unitary $Z$ and a phase string $\eta=(\eta_1,\ldots,\eta_L)$, write
\[
 \Phi_L(Z;\eta)
 =D_{\eta_L}B_N\cdots D_{\eta_1}B_NZ
\]
for the endpoint after $L$ steps.  Draw a product-uniform phase string $\xi$
and set
\[
 X_L=\Phi_L(X;\xi),
 \qquad
 Y_L=\Phi_L(Y;\Psi_h(\xi)).
\]
The first endpoint has law $K_N^L(X,\cdot)$.  Since $\Psi_h$ preserves
product-uniform measure, $\Psi_h(\xi)$ is itself product-uniform, and the
second endpoint has law $K_N^L(Y,\cdot)$.  Thus $(X_L,Y_L)$ is an exact
coupling of the required transition laws for every sufficiently small $h$.
Only the estimate of its distance will use a first-order expansion.

We now make that expansion precise.  For small $h$, the relative unitary
$Y_LX_L^*$ is close to the identity and has a unique small skew-Hermitian
logarithm: the matrix $S$ such that $Y_LX_L^*=e^S$.  Iterating the one-step
calculation above gives
\begin{equation}\label{eq:relative-log-expansion}
 \log(Y_LX_L^*)=hR_L(\xi)+E_L(h,\xi),
 \qquad
 \sup_\xi\|E_L(h,\xi)\|_{\HS}\le C_Nh^2
\end{equation}
when $h$ is sufficiently small.  The bound is uniform in $\xi$ because the
endpoint map, the phase corrections, and their derivatives are continuous
on the compact phase torus, and $L=4N+1$ is fixed once $N$ is fixed.  The
constant may also be chosen uniformly over $X\in\U(N)$ and
$\|A\|_{\HS}=1$, since these parameter spaces are compact.

To see why the error remains quadratic, suppose after step $s-1$ that the
relative logarithm has the form $hR_{s-1}+E_{s-1}$ with
$\|E_{s-1}\|_{\HS}=O_N(h^2)$.  Conjugation by $B_N$ and by a diagonal
unitary preserves the size of this error.  The Baker--Campbell--Hausdorff
expansion used to combine the transported displacement with the phase
correction adds only terms containing at least two factors of size $O(h)$.
Thus the new error is bounded by a constant times
$\|E_{s-1}\|_{\HS}+h^2$.  Starting from zero error and repeating this for
the fixed number $L$ of steps proves \eqref{eq:relative-log-expansion};
compactness makes all of the constants uniform in the phase string.

For nearby unitaries, intrinsic distance is the Hilbert--Schmidt norm of the
principal relative logarithm.  Moreover,
$\|R_L\|_{\HS}\le\|A\|_{\HS}=1$, because each step defining $R_L$ consists
of unitary conjugations and the orthogonal projection $Q$.  It follows from
\eqref{eq:relative-log-expansion} that, uniformly in the phase string,
\[
 \left|D(X_L,Y_L)^2-h^2\|R_L\|_{\HS}^2\right|
 \le C_N'h^3.
\]
Taking expectations and then using \eqref{eq:residual-half} gives
\begin{equation}\label{eq:local-expansion}
 \E D(X_L,Y_L)^2
 =h^2\E\|R_L\|_{\HS}^2+O_N(h^3)
 \le\frac12h^2+O_N(h^3).
\end{equation}

Since Wasserstein distance is the infimum over all couplings, our particular
coupling gives
\[
 W_{2,D}(K_N^L(X,\cdot),K_N^L(Y,\cdot))^2
 \le \frac12h^2+O_N(h^3).
\]
Recall that $h=D(X,Y)$.  Dividing by $h^2$, taking square roots, and letting
$Y\to X$ now gives
\[
 \limsup_{Y\to X}
 \frac{W_{2,D}(K_N^L(X,\cdot),K_N^L(Y,\cdot))}{D(X,Y)}
 \le 2^{-1/2}
\]
uniformly in $X$.

We now use Oliveira's local-to-global path-coupling theorem
\cite[Theorem 3]{Oliveira2009}.  Its content here is simple.  If a Markov
kernel contracts every sufficiently short displacement by the same factor,
and distance is defined by lengths of paths, then subdividing a path into
short pieces extends the same contraction to its two endpoints.  The
technical assumptions are automatic: $(\U(N),D)$ is compact, hence complete
and separable, and $D$ is a path-length metric.  All transition laws also
have finite second moments because the space has finite diameter.

Every sufficiently close pair has the representation used above with
$h=D(X,Y)$.  Equation \eqref{eq:local-expansion} therefore says that
$X\mapsto K_N^L(X,\cdot)$ is locally $2^{-1/2}$-Lipschitz when the source is
measured by $D$ and the target by $W_{2,D}$.  Oliveira's theorem promotes
this local estimate to the global contraction below.  This is the only
external coupling theorem used in the proof.

\begin{proposition}[Block contraction]\label{prop:block-contraction}
For probability measures $\mu,\nu$ on $\U(N)$,
\begin{equation}\label{eq:block-contraction}
 W_{2,D}(\mu K_N^{4N+1},\nu K_N^{4N+1})
 \le2^{-1/2}W_{2,D}(\mu,\nu),
\end{equation}
\end{proposition}

\begin{proof}
Equation \eqref{eq:local-expansion} verifies the infinitesimal hypothesis of
the cited path-coupling theorem with $\gamma=2^{-1/2}$.  The theorem first
gives the pointwise bound
\[
 W_{2,D}(K_N^L(X,\cdot),K_N^L(Y,\cdot))
 \le2^{-1/2}D(X,Y).
\]
Now couple $X\sim\mu$ and $Y\sim\nu$, and, conditional on $(X,Y)$, use a
coupling of their $L$-step transition laws achieving this pointwise bound.
The resulting endpoints have laws $\mu K_N^L$ and $\nu K_N^L$.  Taking the
infimum over the initial coupling of $\mu$ and $\nu$ proves
\eqref{eq:block-contraction}.
\end{proof}

\begin{remark}[A comparison principle for general mixers]
\label{rem:general-mixer-coupling}
The coupling above does not require the dyadic structure of $B_N$ until the
collision estimate is inserted.  To make this explicit, fix any
$B\in\U(N)$ and let $K_B$ be the transition kernel of
\[
 X_{t+1}=D_{\xi_{t+1}}BX_t,
 \]
with the same independent uniform diagonal phases.  Define the associated
unistochastic kernel by
\[
 P_B(a,b)=|B_{ba}|^2.
\]
For two independent $P_B$-chains started from distinct states $a,b$, let
$\tau_B$ be their first meeting time, and put
\[
 q_B(s)=\max_{a\ne b}\Pp_{a,b}\{\tau_B>s\},
 \qquad s\ge0.
\]
Then, for all probability measures $\mu,\nu$ on $\U(N)$,
\begin{equation}\label{eq:general-mixer-contraction}
 W_{2,D}(\mu K_B^{s+1},\nu K_B^{s+1})
 \le \sqrt{q_B(s)}\,W_{2,D}(\mu,\nu).
\end{equation}

Indeed, Lemma~\ref{lem:phase-average} remains valid with $B_N,P_N$ replaced
by $B,P_B$.  If $A$ is an initial tangent displacement and
$Z_1=Q(BAB^*)$, the analogue of \eqref{eq:residual-survival} gives
\[
 \E\|R_{s+1}\|_{\HS}^2
 =\sum_{a\ne b}|(Z_1)_{ab}|^2\Pp_{a,b}\{\tau_B>s\}
 \le q_B(s)\|A\|_{\HS}^2.
\]
The same triangular phase coupling and Taylor expansion therefore give the
local contraction factor $\sqrt{q_B(s)}$, and Oliveira's theorem gives
\eqref{eq:general-mixer-contraction}.  The extra step is the initial
projection that produces $Z_1$.

In particular, if
\[
 h_B=\max_{a\ne b}\E_{a,b}\tau_B<\infty,
 \qquad L_B=\lceil2h_B\rceil+1,
\]
then Markov's inequality gives $q_B(L_B-1)\le1/2$.  Haar stationarity and the
diameter bound in \eqref{eq:metric-compare} consequently yield
\[
 W_{2,\rho_N}(\mu K_B^t,H_N)
 \le \pi\,2^{-\frac12\lfloor t/L_B\rfloor}.
\]
Thus a worst-case meeting-time estimate for two independent copies of the
associated classical chain is sufficient for full-law Wasserstein mixing.
The condition is not automatic: for example, if $B$ is a permutation matrix,
two particles started apart never meet and the comparison gives no strict
contraction.
\end{remark}

\begin{proof}[Proof of Theorem~\ref{thm:full-upper}]
Haar measure is stationary under left multiplication by the random increment
$D_\xi B_N$: multiplying a Haar unitary on the left by any fixed unitary
leaves it Haar, and averaging over the random increment changes nothing.
Write $t=q(4N+1)+r$, with $0\le r<4N+1$, and regard the first $r$
transitions as part of the initial law.  Iterating
Proposition~\ref{prop:block-contraction} for the remaining $q$ blocks gives
\[
 W_{2,D}(\nu_{N,t}^{X_0},H_N)
 \le 2^{-q/2}\diam(\U(N),D).
\]
Finally, \eqref{eq:metric-compare} gives
$\rho_N(U,V)\le D(U,V)/\sqrt N$, while
$\diam(\U(N),D)\le\pi\sqrt N$.  Dividing the last display by $\sqrt N$
yields \eqref{eq:full-upper}.
\end{proof}

\section{Full-law lower bounds}\label{sec:full-lower}

The lower bounds use only the number and regularity of the random phase
parameters; they do not rely on the collision argument.  After $t$ steps the
endpoint is a smooth image of an $Nt$-dimensional phase torus.  We cover that
image at a chosen metric resolution and compare the resulting entropy with
the volume of a Haar ball.  This separates the deterministic geometry of the
reachable set from the probabilistic small-ball estimate.

The comparison has a simple entropy interpretation.  At fixed normalized
resolution, specifying the $Nt$ input phases requires on the order of
$Nt\log t$ bits, whereas covering a positive fraction of the
$N^2$-dimensional unitary group requires order $N^2$ bits.  The first budget
is negligible when $t=o(N/\log N)$, so the reachable set can occupy only a
vanishing Haar neighborhood.  At a fixed \emph{unnormalized} Frobenius
resolution, Haar small balls cost order $N^2\log N$; this stronger volume
decay keeps the obstruction visible up to a sufficiently small constant
multiple of $N$.  The proof below makes these two comparisons quantitative.

\subsection{Entropy of the reachable set}

Let $S_{N,t}$ be the support of $\nu_{N,t}^{X_0}$.  The endpoint map
\[
 \Phi_t:(\T^N)^t\longrightarrow\U(N)
\]
is smooth.  If every scalar phase is changed by at most $\eta$, telescoping
the product gives $\rho_N(\Phi_t(\xi),\Phi_t(\xi'))\le t\eta$.  A coordinate
grid therefore proves the covering estimate
\begin{equation}\label{eq:support-cover}
 N_{\rho_N}(S_{N,t},\delta)
 \le\left\lceil\frac{\pi t}{\delta}\right\rceil^{Nt}.
\end{equation}
Here $N_d(S,\delta)$ is the covering number: the smallest number of open
$d$-balls of radius $\delta$ whose union contains $S$.

We require two standard volume estimates, recorded together for later use.

\begin{lemma}[Haar small balls]\label{lem:haar-small-balls}
For each fixed $0<R<\sqrt2$, there is $c_R>0$ such that, uniformly in
$U\in\U(N)$,
\begin{equation}\label{eq:normalized-small-ball}
 H_N\{V:\rho_N(U,V)\le R\}
 \le \exp\{-c_RN^2+O_R(N)\}.
\end{equation}
For each fixed $r>0$, there is $c_r>0$ such that
\begin{equation}\label{eq:absolute-small-ball}
 H_N\{V:\|U-V\|_{\HS}\le r\}
 \le\exp\{-c_rN^2\log N+O_r(N^2)\}.
\end{equation}
\end{lemma}

\begin{proof}
By invariance take $U=I$.  The identity
\[
 \rho_N(I,V)^2=2-\frac2N\operatorname{Re}\Tr V
\]
turns a normalized small ball of radius $R<\sqrt2$ into a linear-size upper
tail for $\operatorname{Re}\Tr V$.  Exposing the columns of a Haar unitary
successively and applying the complex spherical-cap bound to a positive
fraction of the diagonal coordinates gives \eqref{eq:normalized-small-ball};
the conditional ambient dimensions remain proportional to $N$, so the
product of the cap probabilities is $e^{-c_RN^2}$.  For fixed absolute radius,
the corresponding caps have angular radius $O(N^{-1/2})$, and their product
adds the factor $\log N$ in \eqref{eq:absolute-small-ball}.  A detailed
column-exposure derivation is given in Appendix~\ref{app:small-balls}.
\end{proof}

\subsection{Extremal distance from Haar}

Fix $R<\sqrt2$ and $\delta>0$ with $R+\delta<\sqrt2$.  Cover $S_{N,t}$ by
$\delta$-balls.  Equations \eqref{eq:support-cover} and
\eqref{eq:normalized-small-ball} imply
\begin{equation}\label{eq:tube-mass}
 H_N\{V:\rho_N(V,S_{N,t})\le R\}
 \le \exp\{Nt\log(Ct/\delta)-c_{R+\delta}N^2+O(N)\}.
\end{equation}
If $t=o(N/\log N)$, the right-hand side tends to zero.  Every coupling of a
law supported on $S_{N,t}$ with Haar measure must therefore transport
$1-o(1)$ mass by distance at least $R$.  More explicitly, let $(X,Y)$ have
an arbitrary coupling of $\nu_{N,t}^{X_0}$ and $H_N$.  Then $X\in S_{N,t}$
almost surely, while \eqref{eq:tube-mass} gives
\[
 \Pp\{\rho_N(Y,S_{N,t})>R\}=1-o(1).
\]
On this event, $\rho_N(X,Y)>R$, and hence
\[
 \E\rho_N(X,Y)
 \ge R\,\Pp\{\rho_N(Y,S_{N,t})>R\}
 =R(1-o(1)).
\]
Taking the infimum over all couplings gives
\[
 \liminf_NW_{1,\rho_N}(\nu_{N,t}^{X_0},H_N)\ge R.
\]
Letting $R\uparrow\sqrt2$, and using $W_1\le W_2\le\sqrt2$, proves
\eqref{eq:maximal-lower}.

This is a geometric witness rather than a low-order observable: it tests the
distance of a Haar matrix from the entire reachable set $S_{N,t}$.  That is
why it can distinguish the full law long after one- and two-copy tests have
already equilibrated.

For total variation, $\Phi_t$ maps a real manifold of dimension $Nt$ into the
$N^2$-dimensional manifold $\U(N)$.  If $t<N$, Sard's theorem (or the area
formula) shows that $S_{N,t}$ has Haar measure zero, proving
\eqref{eq:tv-singular}.

Finally fix an absolute Frobenius radius $r>0$.  If each scalar phase is
rounded with mesh $r/(2t\sqrt N)$, telescoping the product changes its endpoint
by at most $r/2$ in Frobenius norm.  Thus
\begin{equation}\label{eq:absolute-support-cover}
 N_{\|\cdot\|_{\HS}}(S_{N,t},r/2)
 \le \left(\frac{Ct\sqrt N}{r}\right)^{Nt}.
\end{equation}
Combining \eqref{eq:absolute-support-cover} with
\eqref{eq:absolute-small-ball}, the Haar mass of the $r$-neighborhood of
$S_{N,t}$ is at most
\[
 \exp\left\{Nt\log\!\left(\frac{Ct\sqrt N}{r}\right)
       -c_{2r}N^2\log N+O_r(N^2)\right\}.
\]
For $t\le c_r'N$ and a sufficiently small constant $c_r'>0$, this tends to
zero.  To obtain the stated bound for a prescribed $\varepsilon>0$, take
$r=2\varepsilon$ and choose $c_\varepsilon=c_{2\varepsilon}'$.  For all
sufficiently large $N$, the Haar mass of the $r$-neighborhood is at most
$1/2$.  Every coupling must then move at least half its mass by distance at
least $r$, and therefore
\[
 W_{1,d_F}(\nu_{N,t}^{X_0},H_N)\ge r/2=\varepsilon.
\]
Since $W_{2,d_F}\ge W_{1,d_F}$, this proves the final claim of
Theorem~\ref{thm:full-lower}.

\section{The first moment}\label{sec:one-design}

The one-copy averaged conjugation channel, or one-copy twirl, asks how an
observable $A$ evolves after averaging over the random walk:
\[
 \cG_{\nu_{N,t}}^{(1)}(A)=\E[X_tAX_t^*].
\]
This section proves both sides of the exact $1$-design statement.  For the
upper bound, phase averaging reduces the problem to the classical kernel
$P_N$.  For the lower bound, we follow one Fourier observable that survives
until the final dyadic scale.

The first average that enters this channel is over the fresh diagonal
phases.  Let
\[
 \cD(A)=\E_\xi D_\xi A D_\xi^*
\]
be this diagonal phase average, also called the diagonal phase twirl.
Entrywise,
\[
 (\cD(A))_{ab}=\one_{a=b}A_{aa},
\]
so $\cD$ simply deletes the off-diagonal entries.  The averaged channel for
one noisy layer is therefore
\[
 T:=\cD\circ\Ad_{B_N},
 \qquad T(A)=\E_\xi[D_\xi B_NA B_N^*D_\xi^*].
\]
After one application of $T$, the output is diagonal.  If
$A=\diag(v)$, then
\[
 (T(A))_{bb}=\sum_a |(B_N)_{ba}|^2v_a
 =\sum_aP_N(a,b)v_a.
\]
Thus, after identifying a diagonal matrix with its diagonal vector, every
subsequent layer applies $P_N^{\mathsf T}$.

\begin{proposition}[Exact $1$-design]\label{prop:one-design}
For $t\ge k+1$,
\[
 \cG_{\nu_{N,t}^{X_0}}^{(1)}(A)=\frac{\Tr A}{N}I
\]
for every $A\in M_N(\C)$ and every deterministic $X_0\in\U(N)$.
\end{proposition}

\begin{proof}
Right multiplication by the fixed $X_0$ only conjugates the input to the
twirl: it replaces $A$ by $X_0AX_0^*$, without changing its trace.  The first
averaged layer produces a diagonal matrix.  By the calculation above, the
next $k$ layers apply $(P_N^{\mathsf T})^k=\Pi$, where
Proposition~\ref{prop:exact-classical} identifies $\Pi$ as averaging all
coordinates.  The resulting diagonal is therefore constant, with value
$\Tr(A)/N$.  This is exactly the Haar one-copy twirl.
\end{proof}

We next prove that this time cannot be improved.  It is convenient to work
with the adjoint channel, because diagonal Fourier matrices then evolve into
one another explicitly.  For a linear map $\Phi$ on matrices, write
$\Phi^*$ for its Hilbert--Schmidt adjoint, characterized by
$\langle \Phi(A),B\rangle_{\HS}=\langle A,\Phi^*(B)\rangle_{\HS}$.
The same argument will also give a lower bound for the two-copy channel.

\begin{proposition}[A surviving Fourier mode]\label{prop:design-lower}
Let
\[
 \Delta_{N,t}^{(m)}=
 \cG_{\nu_{N,t}^{X_0}}^{(m)}-\cG_{H_N}^{(m)},\qquad m=1,2.
\]
For $1\le t\le k$,
\begin{equation}\label{eq:design-lower-beta}
 \|\Delta_{N,t}^{(m)}\|_\diamond
 \ge \prod_{j=0}^{t-2}\left(1-\frac{2^{j+1}}{N}\right)
 \ge c_*:=\prod_{j=1}^{\infty}(1-2^{-j}).
\end{equation}
More generally, for every integer $L\ge1$ and
$t\le k-L+1$, the right-hand side is at least
\begin{equation}\label{eq:design-lower-epsilon}
 c_L:=\prod_{j=L}^{\infty}(1-2^{-j}),
\end{equation}
where $c_L\uparrow1$ as $L\to\infty$.
\end{proposition}

\begin{proof}
Right multiplication of the walk by the deterministic matrix $X_0$ merely
precomposes its twirl by a unitary channel, so the relevant diamond distances
do not depend on $X_0$.  We take $X_0=I$ and write
\[
 T=\cD\circ\Ad_{B_N}
\]
for the averaged one-layer channel.  Its adjoint is
$T^*=\Ad_{B_N^*}\circ\cD$.

For $r\in\Z_N$, let
\[
 Z_r=\diag(e^{2\pi ira/N})_{a\in\Z_N}.
\]
The matrix $Z_r$ is the observable corresponding to the Fourier character
$a\mapsto e^{2\pi ira/N}$ of the position populations.  The adjoint channel
is useful because it propagates this observable backward through one layer.
Using $P_N(a,b)=q_N(b-2a)$ and Lemma~\ref{lem:tent}, direct calculation of
the diagonal gives
\begin{equation}\label{eq:diagonal-fourier-evolution}
 \cD(B_N^*Z_rB_N)=\widehat q_N(r)Z_{2r}.
\end{equation}
Indeed, its $a$th diagonal entry is
\begin{align*}
 (B_N^*Z_rB_N)_{aa}
 &=\sum_{b\in\Z_N}|(B_N)_{ba}|^2e^{2\pi irb/N}\\
 &=\sum_{b\in\Z_N}q_N(b-2a)e^{2\pi irb/N}\\
 &=e^{2\pi ir(2a)/N}\widehat q_N(r),
\end{align*}
which is precisely the $a$th diagonal entry of the right-hand side of
\eqref{eq:diagonal-fourier-evolution}.  The sign convention in the Fourier
transform is immaterial here because $q_N$ is symmetric.

The first application of the adjoint channel gives
\[
 T^*Z_1=B_N^*Z_1B_N.
\]
At the next application, the diagonal projection in $T^*$ acts on this
matrix.  Equation~\eqref{eq:diagonal-fourier-evolution} replaces its
diagonal by $\widehat q_N(1)Z_2$, after which conjugation gives
\[
 (T^*)^2Z_1=\widehat q_N(1)B_N^*Z_2B_N.
\]
Repeating the same calculation doubles the frequency and contributes one
new multiplier at each step.  Induction therefore yields, for $t\ge1$,
\begin{equation}\label{eq:adjoint-fourier-survival}
 (T^*)^tZ_1=
 \left\{\prod_{j=0}^{t-2}\widehat q_N(2^j)\right\}
 B_N^*Z_{2^{t-1}}B_N.
\end{equation}
The empty product for $t=1$ is one.  Thus the frequency doubles after each
layer, while its amplitude is multiplied by the corresponding tent
multiplier.  Since $Z_1$ has trace zero, the adjoint Haar twirl sends it to
zero.  The last matrix in
\eqref{eq:adjoint-fourier-survival} is unitary, and hence has operator norm
one.  Testing the adjoint channel on this one operator gives
\[
 \|\Delta_{N,t}^{(1)}\|_\diamond
 \ge \|\Delta_{N,t}^{(1)}\|_{1\to1}
 =\|(\Delta_{N,t}^{(1)})^*\|_{\infty\to\infty}
 \ge\prod_{j=0}^{t-2}\widehat q_N(2^j).
\]
In the definition \eqref{eq:diamond-definition}, taking a one-dimensional
ancilla gives the ordinary induced trace norm.  This proves the first
inequality.  The equality is the usual duality between the induced trace norm
and induced operator norm.
For $t\le k$, Lemma~\ref{lem:tent} turns the last product into the first
quantity in \eqref{eq:design-lower-beta}.  At $t=k$ it equals
$\prod_{j=1}^{k-1}(1-2^{-j})\ge c_*$.  Stopping $L$ dyadic scales earlier
similarly gives \eqref{eq:design-lower-epsilon}.

It remains to pass the obstruction from one copy to two.  Define the
completely positive, trace-preserving maps
\[
 J(A)=A\otimes I/N,
 \qquad R(C)=\Tr_2 C.
\]
The map $J$ adds a maximally mixed second system, while $R$ discards it.
For every probability measure $\mu$ on $\U(N)$,
\[
 R\circ\cG_\mu^{(2)}\circ J=\cG_\mu^{(1)}.
\]
Both $J$ and $R$ have diamond norm one.  Therefore
$\|\Delta_{N,t}^{(2)}\|_\diamond
\ge\|\Delta_{N,t}^{(1)}\|_\diamond$, completing the proof.
\end{proof}

Proposition~\ref{prop:one-design} and
Proposition~\ref{prop:design-lower} show that the exact $1$-design time is
$k+1$ (at $t=0$ the law is a point mass and $N\ge2$).  Given
$0<\varepsilon<1$, choose $L=L(\varepsilon)$ so that
$c_L>\varepsilon$.  Then neither the one-copy nor the two-copy channel can be
$\varepsilon$-close to Haar before $k-L+2$ steps.  This is the
$k-O_\varepsilon(1)$ lower bound asserted after
Theorem~\ref{thm:design-main}.

\section{The two-copy phase average and its surviving sectors}\label{sec:two-copy}

Sections~\ref{sec:two-copy}--\ref{sec:design-proof} prove the approximate
$2$-design upper bound in Theorem~\ref{thm:design-main}.  The present section
identifies the sectors that survive the two-copy phase average.  Section~\ref{sec:focusing}
reduces their spatial degrees of freedom, Section~\ref{sec:recursions} derives
the resulting Fourier recursions, Section~\ref{sec:weighted} proves their
contraction, and Section~\ref{sec:design-proof} assembles these estimates and
converts them to diamond norm.  The calculation verifying the weighted
contraction is given in Appendix~\ref{app:weighted-verification}.

We next identify the finite Markov chains that govern the two-copy twirl
$\cG_{\nu_{N,t}}^{(2)}$.
The one-copy calculation became classical after phase averaging because only
diagonal matrix entries survived.  On two copies, more information survives:
the same two basis labels may occur in either order.  The bookkeeping is
therefore richer, but the guiding idea is unchanged.  We first identify
exactly what phase averaging retains and then ask how $B_N$ transports
that retained information.

\subsection{Representation-theoretic setup}

Only one elementary piece of representation theory is needed.  A unitary
representation of a group assigns to every group element a unitary linear
map, in a way that respects multiplication.  Here the group is $\U(N)$.  Let
$\cH=\C^N$ be the one-copy Hilbert space; the representation space is the
two-copy Hilbert space
\[
 \cH^{\otimes2}=\C^N\otimes\C^N,
\]
and the representation is the tensor-square action
\[
 \rho_2(U)=U^{\otimes2}=U\otimes U.
\]
A subspace is called invariant if every $U^{\otimes2}$ maps it into itself.
An invariant subspace is irreducible if it has no proper nonzero subspace
that is itself invariant under every $U^{\otimes2}$.

The swap operator $F(u\otimes v)=v\otimes u$ commutes with
$U^{\otimes2}$ for every $U$.  Its $+1$ and $-1$ eigenspaces are therefore
invariant.  They are denoted
\[
 \Sym^2(\C^N)\quad\text{and}\quad\Lambda^2(\C^N),
\]
and are called the symmetric and antisymmetric tensor squares.  Thus
\begin{equation}\label{eq:tensor-square-decomposition}
 \cH^{\otimes2}=\Sym^2(\C^N)\oplus\Lambda^2(\C^N).
\end{equation}
For the full unitary group these two invariant subspaces are irreducible and
inequivalent.  Concretely, there is no nonzero linear map $T$ from one sector
to the other satisfying
$T(U^{\otimes2}v)=U^{\otimes2}T(v)$ for every $U$ and $v$; such a map would
be called an intertwiner.  This fact is the representation-theoretic reason
that Haar averaging treats the two spaces as separate sectors.

We now choose concrete bases.  For $0\le i\le j<N$ and $0\le i<j<N$,
respectively, put
\begin{align*}
 s_{ij}&=\frac{e_i\otimes e_j+e_j\otimes e_i}
 {\sqrt{2(1+\one_{i=j})}},\\
 a_{ij}&=\frac{e_i\otimes e_j-e_j\otimes e_i}{\sqrt2}.
\end{align*}
where $e_0,\ldots,e_{N-1}$ are the standard coordinate vectors.  These are
orthonormal bases of $\Sym^2(\C^N)$ and $\Lambda^2(\C^N)$.  It is useful to
index either basis by an unordered pair, writing $s_\alpha=s_{ij}$ or
$a_\alpha=a_{ij}$ as appropriate.

The particular unitary $B_N^{\otimes2}$ is one member of this representation
and preserves both spaces.  Let
\[
 S_N=\Sym^2(B_N),\qquad A_N=\Lambda^2(B_N),
\]
denote the matrices of its restrictions in the bases just displayed.  In
particular, their matrix entries mean
\begin{equation}\label{eq:restriction-matrix-entries}
 (S_N)_{\alpha\beta}
 =\langle s_\alpha,B_N^{\otimes2}s_\beta\rangle,
 \qquad
 (A_N)_{\alpha\beta}
 =\langle a_\alpha,B_N^{\otimes2}a_\beta\rangle.
\end{equation}
Thus $\beta$ labels the input basis vector and $\alpha$ the output basis
vector.  Squaring the absolute values of these entries defines bistochastic
population kernels, meaning nonnegative matrices whose row and column sums
are one:
\begin{equation}\label{eq:population-kernels}
 P_N^+(\alpha,\beta)=|(S_N)_{\alpha\beta}|^2,
 \qquad
 P_N^-(\alpha,\beta)=|(A_N)_{\alpha\beta}|^2.
\end{equation}
Thus $P_N^+(\alpha,\beta)$, for example, is the transition weight from the
symmetric pair $\beta$ to the symmetric pair $\alpha$ after phase
information is discarded.  The decomposition
\eqref{eq:tensor-square-decomposition} is only a convenient coordinate
decomposition; it does not impose bosonic or fermionic statistics on the
underlying system.

\subsection{Surviving sectors}

This subsection has three steps.  First, the initial phase average projects
the full two-copy operator space onto a much smaller surviving space.
Second, every later noisy layer acts within that space as four independent
finite-dimensional evolutions.  Third, Haar averaging identifies the target
of those evolutions: uniform population in each irreducible sector and no
coherence between the sectors.  This is the link between the two-copy twirl
and the finite chains analyzed in the following sections.

Let $\cD^{(2)}$ denote the two-copy diagonal phase average, or phase twirl,
\[
 \cD^{(2)}(C)=\E_\xi D_\xi^{\otimes2}C(D_\xi^*)^{\otimes2}.
\]
The averaged channel for one complete noisy layer is
\begin{equation}\label{eq:one-layer-two-copy}
 \mathcal T^{(2)}(C)
 =\cD^{(2)}\bigl(B_N^{\otimes2}C(B_N^*)^{\otimes2}\bigr).
\end{equation}
For $X_0=I$, independence of the freshly sampled phases gives
\begin{equation}\label{eq:iterate-one-layer-channel}
 \cG_{\nu_{N,t}}^{(2)}=(\mathcal T^{(2)})^t.
\end{equation}
For a deterministic $X_0$, this becomes the equally explicit formula
\[
 \cG_{\nu_{N,t}^{X_0}}^{(2)}(C)
 =(\mathcal T^{(2)})^t
 \bigl(X_0^{\otimes2}C(X_0^*)^{\otimes2}\bigr).
\]

Let $\mathcal R=\operatorname{ran}\cD^{(2)}$ be the range of the phase
average, namely the set of its possible outputs.  Equation
\eqref{eq:one-layer-two-copy} shows that the first averaged layer sends every
operator into $\mathcal R$.  All later layers therefore use only the
restriction
\[
 \mathcal K=\mathcal T^{(2)}|_{\mathcal R},
 \qquad
 (\mathcal T^{(2)})^t=\mathcal K^{t-1}\mathcal T^{(2)},\quad t\ge1.
\]
Thus the immediate questions are: what information belongs to $\mathcal R$,
and how does $\mathcal K$ move it?

We use the standard rank-one-operator notation
$|v\rangle\langle w|:z\mapsto v(w^*z)$, and abbreviate
$|i,j\rangle=e_i\otimes e_j$.
In the tensor basis,
\[
 \cD^{(2)}(|i,j\rangle\langle a,b|)=0
 \quad\text{unless}\quad \{i,j\}=\{a,b\}
\]
as multisets.  Hence its range is the orthogonal direct sum of
\begin{enumerate}
\item the symmetric populations $|s_{ij}\rangle\langle s_{ij}|$;
\item the antisymmetric populations $|a_{ij}\rangle\langle a_{ij}|$;
\item the coherences $|s_{ij}\rangle\langle a_{ij}|$ and their adjoints,
      for $i<j$.
\end{enumerate}
Indeed, the diagonal phase attached to
$|i,j\rangle\langle a,b|$ is
$e^{i(\xi_i+\xi_j-\xi_a-\xi_b)}$.  Independent uniform averaging kills this
character unless every phase occurs with the same multiplicity on the two
sides, which is exactly the multiset condition above.  This elementary
character test is the two-copy analogue of deleting off-diagonal entries in
the one-copy calculation.

These are precisely the four types of coordinates announced above.  For
example, when $N=4$, the full operator space on two copies has dimension
$4^4=256$, while this phase average retains only $10$ symmetric populations,
$6$ antisymmetric populations, and two sets of $6$ coherences, for a total
of $28$ dimensions.  The reduction is already substantial before any mixing
estimate is used.

We next compute $\mathcal K$ on the corresponding basis operators.
Consider first a symmetric population indexed by an unordered pair $\beta$.
By \eqref{eq:restriction-matrix-entries},
\[
 B_N^{\otimes2}s_\beta
 =\sum_\alpha(S_N)_{\alpha\beta}s_\alpha.
\]
Conjugating $|s_\beta\rangle\langle s_\beta|$ therefore produces terms
$|s_\alpha\rangle\langle s_\gamma|$.  The phase average deletes every term
with $\alpha\ne\gamma$, leaving
\begin{equation}\label{eq:symmetric-population-action}
 \mathcal T^{(2)}
 \bigl(|s_\beta\rangle\langle s_\beta|\bigr)
 =\sum_\alpha |(S_N)_{\alpha\beta}|^2
 |s_\alpha\rangle\langle s_\alpha|.
\end{equation}
Thus symmetric population weights evolve by $P_N^+$.  The identical
calculation with the antisymmetric basis gives
\begin{equation}\label{eq:antisymmetric-population-action}
 \mathcal T^{(2)}
 \bigl(|a_\beta\rangle\langle a_\beta|\bigr)
 =\sum_\alpha |(A_N)_{\alpha\beta}|^2
 |a_\alpha\rangle\langle a_\alpha|,
\end{equation}
so antisymmetric population weights evolve by $P_N^-$.

Now consider a coherence basis operator
$|s_\beta\rangle\langle a_\beta|$.  Conjugation by $B_N^{\otimes2}$ sends
this rank-one operator to
\[
 |B_N^{\otimes2}s_\beta\rangle
 \langle B_N^{\otimes2}a_\beta|.
\]
The two vectors have the expansions
\[
 B_N^{\otimes2}s_\beta
 =\sum_\alpha(S_N)_{\alpha\beta}s_\alpha,
 \qquad
 B_N^{\otimes2}a_\beta
 =\sum_\gamma(A_N)_{\gamma\beta}a_\gamma.
\]
When the second vector is turned into a bra, its scalar coefficients are
complex-conjugated.  Before phase averaging, the resulting operator is
therefore
\[
 \sum_{\alpha,\gamma}(S_N)_{\alpha\beta}
 \overline{(A_N)_{\gamma\beta}}
 |s_\alpha\rangle\langle a_\gamma|.
\]
The phase average retains only terms for which the two output pairs agree,
that is, $\alpha=\gamma$.  Hence, for a pair $\beta$ with distinct entries,
\begin{equation}\label{eq:coherence-action}
 \mathcal T^{(2)}
 \bigl(|s_\beta\rangle\langle a_\beta|\bigr)
 =\sum_\alpha (S_N)_{\alpha\beta}
       \overline{(A_N)_{\alpha\beta}}
       |s_\alpha\rangle\langle a_\alpha|.
\end{equation}
This motivates the coherence matrix
\begin{equation}\label{eq:coherence-matrix}
 C_N(\alpha,\beta)
 =(S_N)_{\alpha\beta}\overline{(A_N)_{\alpha\beta}},
 \qquad \alpha,\beta\in\{(i,j):i<j\}.
\end{equation}
Taking adjoints in \eqref{eq:coherence-action} shows that the reverse
coherences $|a_\beta\rangle\langle s_\beta|$ evolve by the entrywise
conjugate matrix $\overline{C_N}$.

To state the resulting decomposition formally, define
\begin{align*}
 \mathcal R_+
 &=\operatorname{span}\{|s_\alpha\rangle\langle s_\alpha|:\alpha=(i,j),
       \ i\le j\},\\
 \mathcal R_-
 &=\operatorname{span}\{|a_\alpha\rangle\langle a_\alpha|:\alpha=(i,j),
       \ i<j\},\\
 \mathcal C_{+-}
 &=\operatorname{span}\{|s_\alpha\rangle\langle a_\alpha|:\alpha=(i,j),
       \ i<j\},\\
 \mathcal C_{-+}
 &=\operatorname{span}\{|a_\alpha\rangle\langle s_\alpha|:\alpha=(i,j),
       \ i<j\}.
\end{align*}
The range of the phase average is the orthogonal direct sum
\begin{equation}\label{eq:surviving-space-decomposition}
 \mathcal R
 =\mathcal R_+\oplus\mathcal R_-
  \oplus\mathcal C_{+-}\oplus\mathcal C_{-+}.
\end{equation}
Equations \eqref{eq:symmetric-population-action},
\eqref{eq:antisymmetric-population-action}, and
\eqref{eq:coherence-action}, together with the adjoint of the last equation,
show that each of these four subspaces is invariant under $\mathcal K$:
\[
 \mathcal K(\mathcal R_+)\subseteq\mathcal R_+,
 \quad \mathcal K(\mathcal R_-)\subseteq\mathcal R_-,
 \quad \mathcal K(\mathcal C_{+-})\subseteq\mathcal C_{+-},
 \quad \mathcal K(\mathcal C_{-+})\subseteq\mathcal C_{-+}.
\]
In the displayed bases of these subspaces, the four restrictions of
$\mathcal K$ are represented respectively by $P_N^+$, $P_N^-$, $C_N$, and
$\overline{C_N}$.  Equivalently, relative to the direct sum
\eqref{eq:surviving-space-decomposition}, the matrix of $\mathcal K$ is block
diagonal:
\begin{equation}\label{eq:four-blocks}
 P_N^+\oplus P_N^-\oplus C_N\oplus\overline{C_N}.
\end{equation}

We next identify what these four blocks must approach.  Intuitively, Haar
averaging forgets which particular pair carries population, but it cannot
move total mass between the symmetric and antisymmetric subspaces because
every $U^{\otimes2}$ preserves both.  It therefore leaves one uniform
population in each sector.  By contrast, a phase relation between the two
sectors is not invariant under the full unitary group, so Haar averaging
removes it.

The formal argument is the representation-theoretic version of this
intuition.  The Haar two-copy twirl averages the conjugation action
\[
 C\longmapsto U^{\otimes2}C(U^*)^{\otimes2}.
\]
Such averaging is the orthogonal projection onto the \emph{commutant}: the
space of operators that commute with every $U^{\otimes2}$.  Schur's lemma
says that an operator commuting with an irreducible representation must be a
scalar multiple of the identity, and that there is no nonzero intertwining
map between inequivalent irreducible representations.  Applied to
\eqref{eq:tensor-square-decomposition}, it follows that Haar averaging keeps
only a scalar identity on each of the symmetric and antisymmetric sectors and
removes the operators between them.

In population coordinates, the two sector identities are represented by
uniform vectors.  The restriction of the Haar twirl to the range of the
phase average is therefore
\begin{equation}\label{eq:haar-blocks}
 \Pi_+\oplus\Pi_-\oplus0\oplus0,
\end{equation}
where $\Pi_\pm$ are the uniform rank-one projections on the corresponding
pair spaces.  Explicitly, $\Pi_\pm$ replaces a population vector by the
constant vector having the same mean.

Equations \eqref{eq:four-blocks} and \eqref{eq:haar-blocks} now turn the
abstract two-design problem into three finite-chain estimates:
\begin{equation}\label{eq:two-copy-proof-goals}
 P_N^+\longrightarrow\Pi_+,
 \qquad P_N^-\longrightarrow\Pi_-,
 \qquad C_N\longrightarrow0.
\end{equation}
The first two assertions say that the symmetric and antisymmetric
populations approach their uniform laws.  The third says that coherence
between the two inequivalent representations disappears; the
reverse-coherence block $\overline{C_N}$ then disappears as well.  More
precisely, since
$\cG_{\nu_{N,t}}^{(2)}=\mathcal K^{t-1}\mathcal T^{(2)}$ and
$\|\mathcal T^{(2)}\|_{2\to2}=1$, while Haar invariance gives
$\cG_{H_N}^{(2)}=\cG_{H_N}^{(2)}\mathcal T^{(2)}$, the block decompositions
\eqref{eq:four-blocks} and \eqref{eq:haar-blocks} imply, for $t\ge1$,
\begin{align}\label{eq:block-reduction-bound}
 &\|\cG_{\nu_{N,t}}^{(2)}-\cG_{H_N}^{(2)}\|_{2\to2} \notag\\
 &\quad\le
 \max\left\{
 \|(P_N^+)^{t-1}-\Pi_+\|_{2\to2},
 \|(P_N^-)^{t-1}-\Pi_-\|_{2\to2},
 \|C_N^{t-1}\|_{2\to2}
 \right\}.
\end{align}
The omitted reverse-coherence term has the same norm as $C_N^{t-1}$.
Thus proving the three limits in \eqref{eq:two-copy-proof-goals}, with
quantitative bounds, directly bounds the actual $t$-step two-copy twirl.

The next sections establish those bounds in stages.  The coherence matrix
satisfies $C_N^k=0$, so that block vanishes exactly after $k$ restricted
layers.  For each population chain, translation covariance implies that
after $k$ steps its adjoint maps every observable into the
translation-invariant subspace.  Such observables depend only on the
circular separation of the two indices, reducing a chain with order $N^2$
states to a one-dimensional separation chain with order $N$ states.  We
then express this quotient chain in a centered-cosine basis and prove that
the resulting coefficient recursion contracts in a weighted norm by a
factor $\rho<1$ independent of $N$.  Finally, the focusing inequality lifts
that contraction from the separation quotient to the original population
chain: the first $k$ steps place every observable in the
translation-invariant subspace without increasing its $L^2$ norm, and each
subsequent step contracts the remaining separation-dependent component.
Thus the population estimate consists of an exact $k$-step reduction
followed by contraction on the one-dimensional quotient.

The following identity is not needed for the mixing estimate, but it provides
a useful consistency check.  If
\[
 \cF_2(\mu)=\int\!\int|\Tr(U^*V)|^4\,d\mu(U)d\mu(V)
\]
is the second frame potential, then orthogonality of the four sectors and the
fact that the initial phase projection preserves Hilbert--Schmidt norm on its
range give, for $t\ge1$,
\begin{equation}\label{eq:frame-potential}
 \cF_2(\nu_{N,t})
 =\|(P_N^+)^{t-1}\|_{\HS}^2
 +\|(P_N^-)^{t-1}\|_{\HS}^2
 +2\|C_N^{t-1}\|_{\HS}^2.
\end{equation}
In particular, Haar measure has frame potential $2$.

\subsection{The coherence blocks vanish exactly}

The following gauge property is special to the periodic Balazs--Voros
quantization.  It lets us replace $B_N$ by a real matrix without changing
the phase-averaged problem.  For a real matrix, the coherence block can then
be recognized as a matrix of $2\times2$ minors of the one-particle kernel.

\begin{lemma}[Real gauge]\label{lem:real-gauge}
There are diagonal unitaries $L_N,R_N$ for which $O_N=L_NB_NR_N$ is real
orthogonal.  Fix $X_0$, and let $\widetilde\nu_{N,t}^{X_0}$ be the law of the
walk obtained by using $O_N$ in place of $B_N$.  Then, for every $m$ and $t$,
there is a fixed unitary $V$ such that
\begin{equation}\label{eq:real-gauge-twirl}
 \cG_{\widetilde\nu_{N,t}^{X_0}}^{(m)}(A)
 =\cG_{\nu_{N,t}^{X_0}}^{(m)}
   \bigl(V^{\otimes m}A(V^*)^{\otimes m}\bigr).
\end{equation}
Thus the two walks have the same $m$-copy twirl up to precomposition by a
fixed unitary conjugation channel.
\end{lemma}

\begin{proof}
The proof has two parts.  We first choose row and column phases that make
every entry of $B_N$ real.  We then check that inserting those diagonal
phases into the walk changes only a fixed unitary at its initial end.

Write $a=a_0+hM$ as in \eqref{eq:baker-entry} and put
$d=b-2a_0$.  When $d$ is odd, the geometric sum in that formula is
\[
 \sum_{r=0}^{M-1}e^{2\pi ird/N}
 =i e^{-\pi id/N}\csc(\pi d/N),
\]
where the cosecant is real and carries the sign.  Define diagonal matrices
$L_N$ and $R_N$ by multiplying row $b$ and column $a$, respectively, by
\[
 \ell_b=e^{\pi ib/N}(-i)^{b\bmod2},
 \qquad
 r_a=e^{-2\pi ia_0/N}.
\]
For odd $d$, one also has $b\bmod2=1$.  The phase multiplying the geometric
sum is therefore
\[
 e^{\pi i(b-2a_0)/N}(-i)\,i e^{-\pi id/N}=1.
\]
Thus $(L_NB_NR_N)_{ba}$ is real.  If $d=0$, then $b$ is even and the same
conclusion follows directly from the length-$M$ sum.  For every other even
$d$ the sum vanishes.  Hence $O_N=L_NB_NR_N$ is real entry by entry.  It is
also unitary, because it is a product of unitaries, and a real unitary matrix
is orthogonal.

It remains to compare the two walks.  Expanding the walk driven by $O_N$
gives
\[
 \widetilde X_t
 =D_{\xi_t}L_NB_NR_N\cdots
   D_{\xi_1}L_NB_NR_NX_0.
\]
All matrices $D_{\xi_s}$, $L_N$, and $R_N$ are diagonal.  Each interior
product of a fresh uniform diagonal phase with the fixed diagonal factors is
therefore again a fresh uniform diagonal phase; adding fixed angles to
independent uniform angles preserves both uniformity and independence.
After absorbing all such interior factors, only the rightmost $R_N$ remains.
Consequently,
\[
 \widetilde X_t\ \stackrel{d}{=}\ X_tV,
 \qquad V=X_0^*R_NX_0.
\]
Finally, substituting $\widetilde X_t\stackrel d=X_tV$ into
\eqref{eq:twirl-def} gives
\[
 \E\bigl[\widetilde X_t^{\otimes m}A
              (\widetilde X_t^*)^{\otimes m}\bigr]
 =\E\bigl[X_t^{\otimes m}
       \{V^{\otimes m}A(V^*)^{\otimes m}\}
       (X_t^*)^{\otimes m}\bigr],
\]
which is \eqref{eq:real-gauge-twirl}.
\end{proof}

We may consequently work in the real gauge; we continue to write $C_N$ for
the coherence matrix in that gauge.  Fix distinct input indices $a,b$ and
distinct output indices $i,j$, and set
\[
 x=(O_N)_{ia}(O_N)_{jb},\qquad
 y=(O_N)_{ib}(O_N)_{ja},
\]
for the two possible assignments of the input indices to the output indices.
Because $O_N$ is real, the symmetric and antisymmetric amplitudes are
$x+y$ and $x-y$, and the corresponding coherence coefficient is
\[
 (x+y)(x-y)=x^2-y^2.
\]
On the other hand, our convention
$P_N(a,i)=|(O_N)_{ia}|^2=(O_N)_{ia}^2$ gives
\[
 x^2-y^2
 =P_N(a,i)P_N(b,j)-P_N(b,i)P_N(a,j).
\]
This is the $2\times2$ minor of $P_N^{\mathsf T}$ with rows
$\{i,j\}$ and columns $\{a,b\}$.  Since $\Lambda^2(T)$ is represented by
the $2\times2$ minors of $T$, we obtain
\begin{equation}\label{eq:coherence-exterior}
 C_N=\Lambda^2(P_N^{\mathsf T}).
\end{equation}
Exterior powers preserve composition:
$\Lambda^2(T_1T_2)=\Lambda^2(T_1)\Lambda^2(T_2)$.  Therefore
\eqref{eq:Pk-uniform} implies
\begin{equation}\label{eq:coherence-zero}
 C_N^k=\Lambda^2((P_N^{\mathsf T})^k)=\Lambda^2(\Pi)=0.
\end{equation}
This is the simplest of the three goals in
\eqref{eq:two-copy-proof-goals}.  After $k$ steps the one-particle kernel is
the rank-one uniform projection.  Its exterior square is zero because a
rank-one map cannot carry two linearly independent directions.  Thus every
symmetric--antisymmetric coherence disappears exactly at the classical
focusing time; no asymptotic estimate is required.

\section{Translation focusing and separation quotients}\label{sec:focusing}

Our goal in this section is to reduce each population chain from a problem
on order $N^2$ pairs to a problem on order $N$ possible separations.  More
precisely, we will prove that after $k$ steps the backward operator sends
every observable of a pair to an observable that depends only on the
circular separation between the pair's two entries.  This is an exact
dimension reduction; it does not yet prove that the remaining separation
variable has mixed.

Here is the precise phenomenon behind the reduction.  Take two initial pairs
that differ only by a common shift of both entries around the circle.  After
$k$ steps, they give the same expected value for every observable of the
final pair.  The chain has therefore forgotten the initial pair's location.
It has not necessarily forgotten the distance between the two entries,
which is unchanged by a common shift.  We call this first loss of information
\emph{translation focusing}.  The calculation below proves the statement
and identifies pair separation as the coordinate that remains afterward.

For an unordered pair $\alpha=\{i,j\}$, write
$T_x\alpha=\{i+x,j+x\}$, with addition modulo $N$.  Thus $T_x$ moves the
pair around the circle without changing the distance between its two
entries.  The entry formula \eqref{eq:baker-entry}
implies
\begin{equation}\label{eq:pair-covariance}
 P_N^\pm(T_{2x}\alpha,T_x\beta)=P_N^\pm(\alpha,\beta).
\end{equation}
It is useful here to pass from probability distributions to observables.
With the input--output convention of \eqref{eq:population-kernels}, the
backward operator is $Q_N^\pm=(P_N^\pm)^T$; explicitly,
\[
 (Q_N^\pm f)(\beta)=\sum_\alpha P_N^\pm(\alpha,\beta)f(\alpha).
\]
Because the translations form the cyclic group $\Z_N$, every observable
decomposes orthogonally into the Fourier-mode spaces
\[
 \cH_p=\{f:f(T_x\alpha)=e^{2\pi ipx/N}f(\alpha)\},
 \qquad p\in\Z_N.
\]
Thus an observable in $\cH_p$ acquires the phase
$e^{2\pi ipx/N}$ when both entries of its argument are shifted by $x$.
Equation \eqref{eq:pair-covariance} gives
\begin{equation}\label{eq:frequency-doubling}
 Q_N^\pm\cH_p\subseteq\cH_{2p}.
\end{equation}
Indeed, if $f\in\cH_p$, then the change of variables
$\alpha=T_{2x}\gamma$ and \eqref{eq:pair-covariance} give
\[
 \begin{aligned}
 (Q_N^\pm f)(T_x\beta)
 &=\sum_\gamma P_N^\pm(T_{2x}\gamma,T_x\beta)
       f(T_{2x}\gamma)\\
 &=e^{2\pi i(2p)x/N}(Q_N^\pm f)(\beta).
 \end{aligned}
\]
This is the defining transformation rule for $\cH_{2p}$.
Iterating gives
\[
 (Q_N^\pm)^j\cH_p\subseteq\cH_{2^jp}.
\]
Since $N=2^k$, one has $2^kp=0$ in $\Z_N$ for every $p$.  Every mode is
therefore focused into the translation-invariant space $\cH_0$ after $k$
steps.

This completes the exact dimension reduction.  To prove population mixing,
it remains to show that the reduced chain forgets the initial separation in
another $O(\log N)$ steps.  Sections~\ref{sec:recursions} and
\ref{sec:weighted} do so through explicit cosine recursions and a weighted
contraction estimate.

The functions in $\cH_0$ depend only on circular separation
\[
 d(i,j)=\min\{|i-j|,N-|i-j|\}\in\{0,\ldots,M\}.
\]
Indeed, two unordered pairs lie in the same orbit under simultaneous
translation exactly when they have the same separation.  Identifying all
pairs in each such orbit therefore produces the \emph{separation quotient}.
Its state space is $\{1,\ldots,M\}$ in the antisymmetric sector and
$\{0,\ldots,M\}$ in the symmetric sector.  Under the identification
\[
 f(\{i,j\})=g(d(i,j)),
\]
the restriction of $Q_N^\pm$ to $\cH_0$ becomes a Markov operator on the
corresponding separation state space.  We call these restricted operators
the antisymmetric and symmetric separation-quotient operators, and continue
to denote them by $Q_N^-$ and $Q_N^+$, respectively.  This is the promised
reduction from order $N^2$ pair states to order $N$ separation states.

The uniform law on antisymmetric pairs induces
\begin{equation}\label{eq:pi-minus}
 \pi_-(d)=\begin{cases}2/(N-1),&d<M,\\1/(N-1),&d=M.\end{cases}
\end{equation}
For $d<M$, each separation class contains $N$ pairs, whereas the antipodal
class $d=M$ contains only $N/2$; this accounts for the exceptional last
weight.  In the symmetric sector the diagonal class $d=0$ is also present.
The induced law on $0,\ldots,M$ is
\begin{equation}\label{eq:pi-plus}
 \pi_+(d)=\begin{cases}2/(N+1),&d<M,\\1/(N+1),&d=M.\end{cases}
\end{equation}
In either sector, $\pi_\pm$ is the pushforward of the uniform measure on
pairs under the separation map: explicitly,
\[
 \pi_\pm(d)
 =\frac{\#\{\text{pairs in the $\pm$ sector with separation $d$}\}}
        {\#\{\text{pairs in the $\pm$ sector}\}}.
\]
We call this the \emph{class-size measure}, because the weight assigned to a
separation is proportional to the number of pairs in that separation class.

\begin{lemma}[Focusing inequality]\label{lem:focusing}
For every $s\ge0$,
\begin{equation}\label{eq:focusing-bound}
 \|(P_N^\pm)^{k+s}-\Pi_\pm\|_{2\to2}
 \le
 \|((Q_N^\pm)|_{\cH_0})^s-\Pi_\pm\|_{2\to2},
\end{equation}
where the norm on the full pair space uses the uniform measure and the norm
on $\cH_0$ uses the corresponding class-size measure $\pi_\pm$ defined in
\eqref{eq:pi-minus}--\eqref{eq:pi-plus}.
\end{lemma}

\begin{proof}
The proof separates the first $k$ focusing steps from the following $s$
mixing steps.  We use three facts: $(Q_N^\pm)^k$ has range in $\cH_0$,
$Q_N^\pm$ is an $L^2$ contraction, and transposition does not change
singular values.

First, the orthogonal Fourier decomposition into the spaces $\cH_p$, together
with \eqref{eq:frequency-doubling}, shows that
\[
 (Q_N^\pm)^k\cH_p\subseteq\cH_{2^kp}=\cH_0
 \qquad\text{for every }p\in\Z_N.
\]
Hence the entire range of $(Q_N^\pm)^k$ is contained in $\cH_0$.

Second, $Q_N^\pm$ is bistochastic: it fixes constants and preserves the
uniform mean.  Hence the projection onto constants satisfies
$\Pi_\pm Q_N^\pm=Q_N^\pm\Pi_\pm=\Pi_\pm$.  It follows that
\[
 (Q_N^\pm)^{k+s}-\Pi_\pm
 =\bigl((Q_N^\pm)^s-\Pi_\pm\bigr)(Q_N^\pm)^k.
\]
The first factor on the right is applied only to vectors in $\cH_0$, while
the second has norm at most one because a bistochastic Markov operator is an
$L^2$ contraction under the uniform law.  Taking norms therefore gives
\[
 \|(Q_N^\pm)^{k+s}-\Pi_\pm\|_{2\to2}
 \le
 \|((Q_N^\pm)|_{\cH_0})^s-\Pi_\pm\|_{2\to2}.
\]

Finally, $Q_N^\pm=(P_N^\pm)^{\mathsf T}$, and a matrix and its transpose
have the same singular values.  The same inequality thus holds for
$P_N^\pm$, proving \eqref{eq:focusing-bound}.  Under the identification of
$\cH_0$ with functions of separation, the uniform pair measure pushes
forward to $\pi_\pm$, so the norm on the right is precisely the one stated
in the lemma.
\end{proof}

\section{Exact Fourier recursions}\label{sec:recursions}

Our goal in this section is to describe explicitly how the two
separation-quotient operators act on mean-zero observables.  We begin with
the cosine functions
\[
 c_m(d)=\cos(2\pi md/N),\qquad 1\le m\le M,
\]
of the separation $d$.  In each sector we subtract the $\pi_\pm$-mean of
$c_m$; the resulting \emph{centered cosine} $g_m$ is orthogonal to the
constant functions.  These centered cosines span the mean-zero observables
on the corresponding separation space.  Propositions
\ref{prop:anti-recursion} and \ref{prop:sym-recursion} compute each
$Q_N^\pm g_m$ exactly as a linear combination of the centered cosines.  In
this concrete sense they give a centered-cosine representation of the two
operators.  Section~\ref{sec:weighted} will bound the resulting recursions
for their coefficients.

The form of the recursions is more important than their individual
coefficients.  A mode of frequency $m$ has a leading image at the doubled
and folded frequency $2m$, reflecting the underlying baker dynamics.  Lower
frequencies are also created through a Fej\'{e}r-kernel correction.  In the
antisymmetric sector this correction enters with a minus sign and supplies
useful cancellation.  In the symmetric sector it enters with a plus sign and
there is an additional diagonal-pair term.  The weighted norm in the next
section is designed to control all three effects simultaneously.

We next introduce the matrix that makes these recursions computable.  To
evaluate $Q_N^\pm c_m$ at an input pair $\{a,b\}$, we expand the pair
transition probabilities defined in \eqref{eq:population-kernels}.  For
distinct input and output pairs, their amplitudes are
\begin{align*}
 (S_N)_{\{i,j\},\{a,b\}}
   &=(B_N)_{ia}(B_N)_{jb}+(B_N)_{ib}(B_N)_{ja},\\
 (A_N)_{\{i,j\},\{a,b\}}
   &=(B_N)_{ia}(B_N)_{jb}-(B_N)_{ib}(B_N)_{ja}.
\end{align*}
Consequently,
\begin{align*}
 P_N^\pm(\{i,j\},\{a,b\})
 ={}& |(B_N)_{ia}(B_N)_{jb}|^2
      +|(B_N)_{ib}(B_N)_{ja}|^2\\
 &\quad\mathbin{\pm}2\operatorname{Re}\!\left[
 (B_N)_{ia}(B_N)_{jb}
 \overline{(B_N)_{ib}(B_N)_{ja}}\right].
\end{align*}
The first two summands are called the \emph{direct terms}: they correspond
to the two possible assignments of the input labels $a,b$ to the output
labels $i,j$.  The final cross term is called the \emph{exchange term}; it
measures interference between those assignments.  Its sign distinguishes
the symmetric and antisymmetric sectors.

To compute $Q_N^\pm c_m$, we multiply this expression by $c_m(j-i)$ and sum
over output pairs $\{i,j\}$.  The resulting sums can be expressed through
the entries of
\begin{equation}\label{eq:Rm-def}
 Z_m=\diag(e^{2\pi imx/N})_{x\in\Z_N},\qquad
 R_m=B_N^*Z_mB_N,\qquad 1\le m\le M.
\end{equation}
Indeed,
\[
 (R_m)_{ab}
 =\sum_{x\in\Z_N}\overline{(B_N)_{xa}}
      e^{2\pi imx/N}(B_N)_{xb}.
\]
There is a small bookkeeping point because the present sum runs over
distinct unordered output pairs $i<j$, whereas products of entries of $R_m$
naturally sum over all ordered pairs.  For distinct inputs $a,b$, put
\[
 H_{ab}=\sum_i |(B_N)_{ia}|^2|(B_N)_{ib}|^2.
\]
If $D_m(a,b)$ denotes the sum of the two direct terms against
$c_m(j-i)$ over $i<j$, and $E_m(a,b)$ denotes the corresponding exchange
sum with a positive sign, then expansion of the matrix products gives
\begin{align}
 D_m(a,b)
 &=\operatorname{Re}\{(R_m)_{aa}\overline{(R_m)_{bb}}\}-H_{ab},
 \label{eq:direct-Rm}\\
 E_m(a,b)
 &=\frac{|(R_m)_{ab}|^2+|(R_m)_{ba}|^2}{2}-H_{ab}.
 \label{eq:exchange-Rm}
\end{align}
The terms $H_{ab}$ remove the omitted diagonal outputs $i=j$.
Lemma~\ref{lem:Rm-elements} shows that the two squared moduli in
\eqref{eq:exchange-Rm} are equal and computes all the remaining quantities.
In the antisymmetric sector the two $H_{ab}$ corrections cancel between
$D_m-E_m$.  In the symmetric sector, the diagonal output states contribute
$2H_{ab}$, which cancels the correction in $D_m+E_m$.  Thus the diagonal
entries of $R_m$ produce the leading cosine of doubled frequency, while its
off-diagonal squared modulus produces the lower-frequency correction.

For later use, set
\begin{equation}\label{eq:Fejer-def}
 F_{m,M}(d)=\frac1{M^2}
 \left|\sum_{r=0}^{m-1}e^{2\pi ird/M}\right|^2.
\end{equation}
This is a finite Fej\'{e}r kernel.  Its cosine expansion will make the
lower-frequency correction completely explicit.

\begin{lemma}[Fourier-block matrix elements]\label{lem:Rm-elements}
For every $a\in\Z_N$ and every nonzero $d\in\Z_N$,
\begin{align}
 (R_m)_{aa}&=(1-m/M)e^{4\pi ima/N},\label{eq:Rm-diagonal}\\
 |(R_m)_{a,a+d}|^2
 &=\frac1{M^2}\left|\sum_{r=0}^{m-1}e^{2\pi ird/M}\right|^2
 =F_{m,M}(d).\label{eq:Rm-offdiagonal}
\end{align}
Here the second coordinate is interpreted modulo $N$.
\end{lemma}

\begin{proof}
We compute $R_m$ in three steps.  First, conjugating $Z_m$ by the full
Fourier transform turns multiplication by a Fourier character into a shift.
Indeed, put $C=\diag(F_M,F_M)$, so that $B_N=F_N^*C$.  With the Fourier
convention fixed in Section~\ref{sec:classical},
\[
 F_NZ_mF_N^*=S_m,
\]
where $S_m$ cyclically shifts the standard basis by $m$.  Therefore
\[
 R_m=C^*S_mC.
\]

Second, we identify the sums appearing in this product.  The two diagonal
blocks of $C$ divide the coordinates into two intervals of length $M$.
Because $1\le m\le M$, shifting either interval by $m$ leaves an overlap of
length $M-m$ with the same block and moves an interval of length $m$ into
the other block.  Substituting the entries of $F_M$ into $C^*S_mC$ therefore
expresses every entry of $R_m$ as $M^{-1}$ times a geometric sum over one of
these two intervals.

Third, we evaluate the diagonal and off-diagonal cases.  For a diagonal
entry, the $M-m$ summands have a common phase.  Their number supplies the
factor $1-m/M$, and their common phase is $e^{4\pi ima/N}$, which proves
\eqref{eq:Rm-diagonal}.  For an off-diagonal entry with displacement $d$, the
two overlap intervals are complementary inside a complete sum of $M$ roots
of unity.  If $d\not\equiv0\pmod M$, that complete sum vanishes, so the
length-$(M-m)$ sum is, up to a phase, the negative of the complementary
length-$m$ sum.  The phase disappears upon taking absolute values, leaving
\[
 M^{-1}\left|\sum_{r=0}^{m-1}e^{2\pi ird/M}\right|,
\]
for $|(R_m)_{a,a+d}|$.

The only remaining nonzero displacement in $\Z_N$ is the antipodal case
$d=M$.  Then the two columns have the same within-block Fourier coordinate
but lie in opposite blocks of $C$.  Their cross-block overlap in
$C^*S_mC$ has length $m$, and all its summands have the same phase.  Hence
$|(R_m)_{a,a+M}|=m/M$.  This agrees with the right-hand side of
\eqref{eq:Rm-offdiagonal}, since every term in its geometric sum equals one
when $d=M$.  Squaring completes the proof.
\end{proof}

\subsection{Antisymmetric quotient}

For $1\le m\le M$, center the cosine modes with respect to $\pi_-$ by
\[
 g_m=c_m+\frac1{N-1},\qquad 1\le d\le M.
\]
The quotient chain mixes constants trivially, so we work on its mean-zero
subspace.  The added constant makes $g_m$ centered in $L^2(\pi_-)$.  Because
cosines do not distinguish $m$ from $N-m$, frequencies above $M$ are always
folded back into $\{0,\ldots,M\}$.  Put
$\alpha_m=(1-m/M)^2$.

\begin{proposition}[Antisymmetric recursion]\label{prop:anti-recursion}
With frequencies folded modulo $N$ by $g_j=g_{N-j}$ and $g_0=0$,
\begin{align}
 Q_N^-c_m&=\alpha_mc_{2m}-F_{m,M},\label{eq:anti-uncentered}\\
 Q_N^-g_m&=\alpha_mg_{2m}
 -\frac2{M^2}\sum_{\ell=1}^{m-1}(m-\ell)g_{2\ell}.
 \label{eq:anti-centered}
\end{align}
\end{proposition}

\begin{proof}
We first derive the uncentered formula and then rewrite it in the centered
basis.  Use the direct--exchange expansion of $P_N^-$ given above.  For an
input pair $\{a,b\}$, equations \eqref{eq:direct-Rm} and
\eqref{eq:exchange-Rm} give
\[
 D_m(a,b)=\alpha_mc_{2m}(b-a)-H_{ab},
 \qquad
 E_m(a,b)=F_{m,M}(b-a)-H_{ab}.
\]
Here \eqref{eq:Rm-diagonal} produces both the coefficient
$\alpha_m=(1-m/M)^2$ and the doubled frequency $2m$.  The exchange term has
the antisymmetric minus sign, so the total contribution is $D_m-E_m$; the
two $H_{ab}$ terms cancel.  Since the input
separation is $d=b-a$ up to the harmless cosine symmetry $d\leftrightarrow
N-d$, we obtain
\[
 Q_N^-c_m=\alpha_mc_{2m}-F_{m,M},
\]
which is \eqref{eq:anti-uncentered}.

It remains to center this identity.  The finite Fej\'{e}r expansion is
\[
 F_{m,M}(d)=\frac{m}{M^2}
 +\frac2{M^2}\sum_{\ell=1}^{m-1}(m-\ell)
          \cos(2\pi\ell d/M)
\]
and $\cos(2\pi\ell d/M)=c_{2\ell}(d)$.  Because
$g_j=c_j+1/(N-1)$ and $Q_N^-$ fixes constants, substitution into the
uncentered formula expresses every nonconstant term in the $g$-basis.  The
remaining scalar terms cancel by
\[
 1-\alpha_m+\frac{m(m-1)}{M^2}
 =\frac{m(N-1)}{M^2}.
\]
The result is exactly \eqref{eq:anti-centered}.  If a doubled frequency
exceeds $M$, the identity $c_j=c_{N-j}$ folds it back as stated in the
proposition.
\end{proof}

The centered cosines have the single linear dependence
\begin{equation}\label{eq:anti-dependence}
 2\sum_{m=1}^{M-1}g_m+g_M=0
\end{equation}
and Gram matrix, meaning the matrix of their pairwise inner products,
\begin{equation}\label{eq:anti-gram}
 \langle g_m,g_n\rangle_{\pi_-}
 =\frac{M(1+\one_{m=M})}{2M-1}\one_{m=n}
 -\frac{2M}{(2M-1)^2}.
\end{equation}
Both follow from discrete cosine orthogonality after restoring the missing
endpoint $d=0$.  Although the centered cosines are not mutually orthogonal,
equations \eqref{eq:anti-dependence}--\eqref{eq:anti-gram} show that they
form a tight frame for the mean-zero antisymmetric separation space.
Consequently, bounds on the individual images $Tg_m$ can be combined into a
Hilbert--Schmidt, and hence operator-norm, bound for $T$.  The exact frame
identity is used in Section~\ref{sec:weighted}.

\subsection{Symmetric quotient}

For $0\le d\le M$, put
\[
 g_m(d)=\cos(2\pi md/N)-\frac1{N+1},\qquad1\le m\le M.
\]

\begin{proposition}[Symmetric recursion]\label{prop:sym-recursion}
The centered symmetric separation operator satisfies
\begin{align}
 Q_N^+g_m={}&(1-m/M)^2g_{2m}
 +\frac2{M^2}\sum_{\ell=1}^{m-1}(m-\ell)g_{2\ell}\notag\\
 &-\frac{m^2}{M^3}\sum_{j=1}^{M-1}g_j
 -\frac{m^2}{2M^3}g_M.\label{eq:sym-centered}
\end{align}
The $M$ functions $g_1,\ldots,g_M$ form a basis of the mean-zero quotient
space, and
\begin{equation}\label{eq:sym-gram}
 \langle g_m,g_n\rangle_{\pi_+}
 =\frac{M(1+\one_{m=M})}{2M+1}\one_{m=n}
 +\frac{2M}{(2M+1)^2}.
\end{equation}
\end{proposition}

\begin{proof}
There are three points to verify: the contribution of pairs with distinct
entries, the additional contribution of diagonal pairs, and the basis
properties of the centered cosines.

For a distinct input pair $\{a,b\}$, equations \eqref{eq:direct-Rm} and
\eqref{eq:exchange-Rm} show that the contribution from distinct output pairs
is
\[
 D_m(a,b)+E_m(a,b)
 =\alpha_mc_{2m}(b-a)+F_{m,M}(b-a)-2H_{ab}.
\]
The symmetric diagonal output states contribute $2H_{ab}$, so the total over
all symmetric outputs is
$\alpha_mc_{2m}(b-a)+F_{m,M}(b-a)$.  Thus symmetrization changes the sign of
the exchange term relative to the antisymmetric sector, and the Fej\'{e}r
correction is added rather than subtracted.  After replacing $c_j$ by its
centered version $g_j$, these distinct-input contributions give the first
line of \eqref{eq:sym-centered}.

It remains to determine the image at the diagonal input class.  A calculation
with the normalized vector $s_{ii}=e_i\otimes e_i$ gives
\[
 (Q_N^+c_m)(0)=\left(1-\frac mM\right)^2.
\]
By contrast, extending the distinct-pair expression to $d=0$ would give
$(1-m/M)^2+F_{m,M}(0)$, and \eqref{eq:Fejer-def} gives
$F_{m,M}(0)=(m/M)^2$.  The correction required at the diagonal class is
therefore
\[
 -\left(\frac mM\right)^2\one_{d=0}.
\]
To express this term in the same cosine coordinates, use the type-I discrete
cosine transform (DCT-I) identity
\[
 \one_{d=0}=\frac{c_0+c_M}{2M}+\frac1M\sum_{j=1}^{M-1}c_j,
\]
valid on $d=0,\ldots,M$.  Multiplication by $-(m/M)^2$ supplies the
coefficients
\[
 -\frac{m^2}{M^3}\quad(1\le j<M),
 \qquad
 -\frac{m^2}{2M^3}\quad(j=M).
\]
After replacing the cosines by their centered versions, the constant part
cancels because $Q_N^+g_m$ has zero $\pi_+$-mean.  This gives the second line
of \eqref{eq:sym-centered}.

Finally, DCT-I orthogonality, with the endpoint weights encoded by $\pi_+$,
gives the inner products in \eqref{eq:sym-gram}.  The resulting Gram matrix
has smallest eigenvalue at least $M/(2M+1)$, so
$g_1,\ldots,g_M$ are linearly independent.  There are $M$ of them and the
mean-zero subspace of functions on $\{0,\ldots,M\}$ has dimension $M$;
hence they form a basis.  The same formula bounds the largest eigenvalue by
an absolute constant, so this basis is uniformly well conditioned in
$L^2(\pi_+)$.
\end{proof}

\section{Weighted Fourier contraction}\label{sec:weighted}

The formulas in Propositions~\ref{prop:anti-recursion} and
\ref{prop:sym-recursion} describe the action of the separation operators on
cosine coefficients.  In the antisymmetric case, the first application
lands in the span of the even frequencies; we compress and relabel those
frequencies as described below.  In the symmetric case, we retain the full
cosine basis.  In either case, let $S_L$ denote the resulting coefficient
matrix, where $L$ is the number of retained frequency coordinates.  Our
goal is to find a norm for which $S_L$ contracts by a factor strictly less
than one, uniformly in $L$.

For $1\le r\le L$, define
\begin{equation}\label{eq:weight}
 w_L(r)=r^{-1/2}\{1+3(r/L)^2\},
 \qquad
 \|a\|_{1,w}=\sum_{r=1}^Lw_L(r)|a_r|.
\end{equation}

The induced matrix norm is
\begin{equation}\label{eq:weighted-column-norm}
 \|S\|_{1,w}
 =\max_{1\le r\le L}\frac{1}{w_L(r)}
   \sum_{s=1}^Lw_L(s)|S(s,r)|.
\end{equation}
The unweighted column norm does not yield a contraction bounded uniformly
away from one.  The factor $r^{-1/2}$ assigns greater weight to low
frequencies, while $1+3(r/L)^2$ increases the weight near the endpoint at
which frequencies fold.  With these two corrections, the absolute column
sums in \eqref{eq:weighted-column-norm} admit a uniform bound below one for
both recursions.  The constant $3$ has no intrinsic significance; it is a
choice for which the two required bounds hold simultaneously.

Because the recursion coefficients have signs, we prove the stronger
entrywise estimate for a nonnegative matrix $A_L$ satisfying
$|S_L(s,r)|\le A_L(s,r)$.  If $\|A_L\|_{1,w}\le\rho<1$, then
$\|S_La\|_{1,w}\le\rho\|a\|_{1,w}$ for every coefficient vector $a$.
The proof establishes this estimate by rescaling $r=xL$: the weighted column
sums converge uniformly to explicit functions $R_0(x)$ and $R_+(x)$, whose
suprema are strictly below one.  Uniform power-sum estimates then transfer
the limiting gap back to the finite matrices.

\begin{lemma}[Weighted contraction]\label{lem:weighted-contraction}
There are $\rho_-,\rho_+<1$ and $L_0<\infty$ such that the following hold
whenever the relevant quotient length is at least $L_0$.
\begin{enumerate}
\item After compressing the image of \eqref{eq:anti-centered} to its even
frequencies, the positive majorant of the coefficient matrix has weighted
$1$-norm at most $\rho_-$.
\item The absolute-value coefficient matrix of \eqref{eq:sym-centered} has
weighted $1$-norm at most $\rho_+$.
\end{enumerate}
\end{lemma}

\begin{proof}
Let $\mathcal C_L^-(r)$ and $\mathcal C_L^+(r)$ denote the weighted column
sums of the two nonnegative matrices in the statement.  The detailed
calculation in Appendix~\ref{app:weighted-verification} constructs explicit
functions $R_0,R_+:[0,1]\to\mathbb R$ and proves
\[
 \sup_{1\le r\le L}
 \left|\mathcal C_L^-(r)-R_0(r/L)\right|\longrightarrow0,
 \qquad
 \sup_{1\le r\le L}
 \left|\mathcal C_L^+(r)-R_+(r/L)\right|\longrightarrow0.
\]
The same calculation gives
\[
 \sup_{0\le x\le1}R_0(x)\le\frac45,
 \qquad
 \sup_{0\le x\le1}R_+(x)<1.
\]
Choose $\rho_-$ and $\rho_+$ strictly between the corresponding suprema and
one.  The uniform convergence then gives the two asserted matrix-norm bounds
for all sufficiently large $L$.
\end{proof}

\subsection{Quotient and full population bounds}

We now convert the coefficient contraction of
Lemma~\ref{lem:weighted-contraction} into the operator estimate needed in
\eqref{eq:two-copy-proof-goals}.  The conversion has three stages.  We first
use the tight-frame structure of the antisymmetric cosines, then use the
uniformly conditioned cosine basis in the symmetric sector, and finally
apply the focusing inequality to return from the separation quotients to the
original pair chains.

\smallskip
\noindent\emph{The antisymmetric quotient.}
Before compression to the even
frequencies, the absolute sum of the coefficients in the image of a single
mode is, by \eqref{eq:anti-centered}, at most
\[
 (1-m/M)^2+\frac{m(m-1)}{M^2}\le1.
\]
Moreover,
\[
 \max_r w_L(r)\le4,\qquad \min_r w_L(r)\ge L^{-1/2}.
\]
For every coefficient vector $a$, it follows that
\[
 (\min_r w_L(r))\|a\|_1
 \le \|a\|_{1,w}
 \le (\max_r w_L(r))\|a\|_1.
\]
Thus passing from the weighted norm to the ordinary coefficient $\ell^1$
norm costs at most $O(\sqrt L)=O(\sqrt N)$.  The first, uncompressed
application has coefficient $\ell^1$ norm at most one by the preceding
column-sum bound; every subsequent compressed application contracts by
$\rho_-$ in the weighted norm.  Finally, the Gram formula
\eqref{eq:anti-gram} bounds the $L^2(\pi_-)$ norm of a cosine combination by
a constant times the $\ell^1$ norm of its coefficients.  Absorbing the first
application and fixed constants into $C$ gives
\begin{equation}\label{eq:anti-mode-bound}
 \|(Q_N^-)^sg_m\|_{L^2(\pi_-)}\le C\sqrt N\,\rho_-^s.
\end{equation}
This controls one centered cosine at a time.  To obtain an operator bound,
use the fact, following from \eqref{eq:anti-dependence}--\eqref{eq:anti-gram},
that the centered cosines form a tight frame for the mean-zero antisymmetric
separation space.  Consequently every operator $T$ on that space satisfies
\begin{equation}\label{eq:anti-tight-frame}
 \|T\|_{\HS}^2=\frac{N-1}{N}
 \left\{2\sum_{m=1}^{M-1}\|Tg_m\|_2^2+\|Tg_M\|_2^2\right\}.
\end{equation}
Take $T=(Q_N^-)^s-\Pi_-$.  Since every $g_m$ is centered,
$Tg_m=(Q_N^-)^sg_m$.  There are $M=O(N)$ modes, so
\eqref{eq:anti-mode-bound} and \eqref{eq:anti-tight-frame} give
\[
 \|T\|_{\HS}^2
 \le C\,N\bigl(\sqrt N\,\rho_-^s\bigr)^2
 =C N^2\rho_-^{2s}.
\]
The operator norm is at most the Hilbert--Schmidt norm.  After changing $C$,
we conclude that
\begin{equation}\label{eq:anti-quotient-bound}
 \|(Q_N^-)^s-\Pi_-\|_{2\to2}\le CN\rho_-^s.
\end{equation}

\smallskip
\noindent\emph{The symmetric quotient.}
Here
$g_1,\ldots,g_M$ form a basis rather than a redundant tight frame.  Define
the synthesis map
\[
 G:\C^M\longrightarrow L_0^2(\pi_+),
 \qquad Ga=\sum_{m=1}^M a_mg_m,
\]
where $L_0^2(\pi_+)$ denotes the mean-zero subspace, and let $S=GG^*$ be its
frame operator.  The Gram formula \eqref{eq:sym-gram} says that the
eigenvalues of $G^*G$ are bounded above and below by absolute positive
constants; in particular, $S\ge I/3$.  Hence $S$ is invertible and
\[
 I=GG^*S^{-1}
\]
on the mean-zero subspace.

Take $T=(Q_N^+)^s-\Pi_+$.  Multiplying the preceding identity by $T$ gives
\[
 T=(TG)G^*S^{-1}.
\]
The weighted contraction controls the $m$th column of $TG$, namely
$Tg_m=(Q_N^+)^sg_m$, and gives
\[
 \|Tg_m\|_2\le C\sqrt N\rho_+^s
 \qquad(1\le m\le M).
\]
Summing over the $M=O(N)$ columns yields
$\|TG\|_{\HS}\le CN\rho_+^s$.  The uniform Gram bounds also give uniform
bounds on $\|G^*\|$ and $\|S^{-1}\|$.  Taking the operator norm in the
factorization of $T$ therefore proves
\begin{equation}\label{eq:sym-quotient-bound}
 \|(Q_N^+)^s-\Pi_+\|_{2\to2}\le CN\rho_+^s.
\end{equation}

\smallskip
\noindent\emph{Return to the full population chains.}
We have now bounded both separation-quotient operators on their mean-zero
subspaces.  Lemma~\ref{lem:focusing} says that the preceding $k$ focusing
steps have norm at most one and place every observable in precisely these
subspaces.  Combining \eqref{eq:anti-quotient-bound},
\eqref{eq:sym-quotient-bound}, and
Lemma~\ref{lem:focusing}, we obtain constants $C<\infty$ and $\rho<1$ such
that
\begin{equation}\label{eq:full-population-bound}
 \max_\pm\|(P_N^\pm)^{k+s}-\Pi_\pm\|_{2\to2}
 \le CN\rho^s.
\end{equation}

The estimate has a transparent division of labor.  The first $k$ steps remove
all dependence on the pair's common location exactly.  Each of the next $s$
steps contracts the remaining separation information by the uniform factor
$\rho$.  The prefactor $N$ comes from reconstructing an operator estimate
from the cosine coordinates.  It costs only another constant multiple of
$\log N$ steps and therefore does not change the Ehrenfest-scale conclusion.

\section{Proof of the two-design theorem}\label{sec:design-proof}

We now assemble the preceding estimates.  We first take $X_0=I$ and write
$\nu_{N,t}$ for the law of the walk after $t$ steps.  Its two-copy twirling
channel is
\[
 \cG_{\nu_{N,t}}^{(2)}(C)
 =\E\left[X_t^{\otimes2}C(X_t^*)^{\otimes2}\right].
\]
The proof has four steps: factor this channel through the sectors that
survive phase averaging, apply the four block estimates, convert the
resulting Hilbert--Schmidt operator bound to diamond norm, and finally choose
the number of steps.

\smallskip
\noindent\emph{Step 1: factor through the surviving space.}
Recall from \eqref{eq:one-layer-two-copy} that $\mathcal T^{(2)}$ is the
first averaged layer and from Section~\ref{sec:two-copy}
that $\mathcal K$ is its restriction to the surviving space
$\mathcal R=\operatorname{ran}\cD^{(2)}$.  Equation
\eqref{eq:iterate-one-layer-channel} gives
\[
 \cG_{\nu_{N,t}}^{(2)}=\mathcal K^{t-1}\mathcal T^{(2)}.
\]
The factorization separates entry into the surviving phase sectors from
subsequent evolution within them.  The first layer $\mathcal T^{(2)}$ performs the phase
projection described in Section~\ref{sec:two-copy}; every later layer is the
same four-block operator $\mathcal K$.  Thus no information outside those blocks can
reappear.
In particular, the exponent $t-1$ counts the layers that remain after this
initial projection.

\smallskip
\noindent\emph{Step 2: apply the four block estimates.}
The map $\mathcal T^{(2)}$ is a unitary conjugation followed by an orthogonal
projection, so $\|\mathcal T^{(2)}\|_{2\to2}=1$.  Haar invariance gives
$\cG_{H_N}^{(2)}=\cG_{H_N}^{(2)}\mathcal T^{(2)}$, and
\eqref{eq:haar-blocks} identifies its restriction to $\mathcal R$.  Hence
\[
 \cG_{\nu_{N,t}}^{(2)}-\cG_{H_N}^{(2)}
 =\left[
   \mathcal K^{t-1}
   -(\Pi_+\oplus\Pi_-\oplus0\oplus0)
  \right]\mathcal T^{(2)}.
\]
The operator in brackets is block diagonal.  Its norm is therefore the
largest norm of its four blocks.  The two coherence blocks vanish after $k$
restricted layers by \eqref{eq:coherence-zero}.  After $k+s$ restricted
layers, \eqref{eq:full-population-bound} places each population block within
$CN\rho^s$ of its stationary projection.  Thus, whenever $t-1\ge k+s$,
\begin{equation}\label{eq:twirl-2to2}
 \|\cG_{\nu_{N,t}}^{(2)}-\cG_{H_N}^{(2)}\|_{2\to2}
 \le CN\rho^s.
\end{equation}

\smallskip
\noindent\emph{Step 3: convert to diamond norm.}
The two-copy Hilbert space has dimension $D=N^2$.
For a linear map $\Phi:M_D(\C)\to M_D(\C)$, an ancillary space of dimension $D$
suffices in the definition of the diamond norm.  For any input
$X\in M_D(\C)\otimes M_D(\C)$, the output acts on a space of dimension
$D^2$.  The Schatten-norm comparison and stability of the induced
Hilbert--Schmidt norm give
\[
 \begin{aligned}
 \|(\Phi\otimes\mathrm{id}_D)(X)\|_1
 &\le D\|(\Phi\otimes\mathrm{id}_D)(X)\|_2\\
 &\le D\|\Phi\|_{2\to2}\|X\|_2
 \le D\|\Phi\|_{2\to2}\|X\|_1.
 \end{aligned}
\]
Taking the supremum over $X$ yields
\begin{equation}\label{eq:diamond-conversion}
 \|\Phi\|_\diamond\le D\|\Phi\|_{2\to2}=N^2\|\Phi\|_{2\to2}.
\end{equation}
This conversion is deliberately crude but sufficient.  Together with the
factor $N$ in \eqref{eq:twirl-2to2}, it asks the exponential contraction
$\rho^s$ to overcome a polynomial loss $N^3$.  Taking logarithms explains
the coefficient $3\log N$ in the next display and, more importantly, why the
total time remains of order $\log N$.

\smallskip
\noindent\emph{Step 4: choose the time and restore $X_0$.}
Choose
\[
 s\ge\frac{3\log N+\log(C/\varepsilon)}{|\log\rho|}.
\]
Then $CN^3\rho^s\le\varepsilon$.  Since the first projection layer is
followed by $k+s$ restricted layers, the total time is
$t=1+k+s=O(\log N+\log\varepsilon^{-1})$.  Equations
\eqref{eq:twirl-2to2}--\eqref{eq:diamond-conversion} prove the two-copy upper
bound in Theorem~\ref{thm:design-main}.  The exact one-copy
cutoff and both lower bounds are Propositions~\ref{prop:one-design} and
\ref{prop:design-lower}.  Finally, restoring a deterministic $X_0$ merely
precomposes the walk twirl and the corresponding Haar comparison with a
unitary channel.  Such composition preserves diamond norm, so the same bound
holds uniformly for every $X_0$.

\section{Discussion and open questions}\label{sec:discussion}

The hierarchy \eqref{eq:hierarchy} separates several questions that are often
grouped under the informal phrase ``mixing to Haar.''  Exact one-copy mixing
and fixed-accuracy two-copy mixing occur on the scale $k=\log_2N$.  At every
$t=o(N/\log N)$, however, the complete distribution is still asymptotically
as far from Haar as the normalized metric permits; it is not even absolutely
continuous for $t<N$.  Thus low-dimensional representations of the walk
forget the generating phase history long before the state space has been
explored at metric resolution.

After the logarithmic window the ensemble therefore gives the Haar answer for
every one-copy conjugation average and, within the design error, for every
two-copy observable.  In the quantum interpretation this includes
ensemble-averaged purity and out-of-time-order expressions of bidegree at most
$(2,2)$, meaning expressions containing at most two factors each of $U$ and
$U^*$.  It does not imply agreement for high-degree spectral statistics.
The lower bound for the full distribution is of a different nature.  It does
not come from a low-degree observable; it comes from the fact that a Haar
unitary is typically far from the set of matrices reachable by the walk in
$t$ steps.  The two time scales therefore describe different levels of
randomization: equilibration of low-order averaged observables and geometric
coverage of $\U(N)$.

Our results concern fresh, spatially independent phase kicks of full strength.
They do not establish a corresponding hierarchy for the deterministic baker,
for weak or temporally correlated noise, or for a local many-body circuit.
The full-law Wasserstein comparison itself extends to any fixed mixer whose
associated unistochastic chain has a suitable two-particle meeting bound, as
Remark~\ref{rem:general-mixer-coupling} explains.  Extending the two-copy
analysis, and hence the full hierarchy, would additionally require
replacements for the baker-specific finite-time focusing of translation
characters and the contractive separation quotient.

In particular, the binary-cell perturbations introduced by Schack and Caves
\cite{SchackCaves1993} fall outside the present argument.  Their diagonal
phases are discrete and spatially correlated, so phase averaging is no longer
an exact dephasing projection and the measure-preserving infinitesimal phase
coupling used in Section~\ref{sec:full-upper} is unavailable.  Determining
whether the hierarchy persists for that physically motivated weak-noise model
is therefore a separate problem rather than a direct extension of our proof.

The main sharp-time question left open is whether the two population chains
are already close to equilibrium immediately after the first focusing block.
The present lower and upper bounds determine the fixed-accuracy design time
up to constant factors, but not its leading constant or whether there is a
sharp finite-size transition (a cutoff).
Resolving that question would require substantially sharper control of
$(P_N^\pm)^k-\Pi_\pm$ than is needed for the results proved here.

\appendix

\section{Haar small-ball estimates}\label{app:small-balls}

We prove the two estimates in Lemma~\ref{lem:haar-small-balls}.  Haar
invariance reduces both to balls centered at the identity.  The argument has
three steps.  We first bound a spherical cap for one column conditioned on
the columns already exposed.  A normalized-radius ball then forces a fixed
positive fraction of large diagonal entries, while a fixed unnormalized
radius forces order $N$ diagonal entries into caps whose probabilities
shrink with $N$.

\smallskip
\noindent\emph{The conditional cap bound.}
Suppose that $u$ is uniform on
the unit sphere of a complex subspace $V\subset\C^N$ of dimension $d$, and
let $e$ be a fixed unit vector.  Since
$\langle e,u\rangle=\langle P_Ve,u\rangle$ and $\|P_Ve\|\le1$, for
$0<b<1$,
\begin{equation}\label{eq:complex-cap}
 \Pp\{\operatorname{Re}\langle e,u\rangle\ge b\}
 \le\Pp\{|\langle v,u\rangle|\ge b\}
 =(1-b^2)^{d-1},
\end{equation}
where $v=P_Ve/\|P_Ve\|$ when $P_Ve\ne0$; if $P_Ve=0$, the left side
vanishes.  By rotational invariance, every fixed unit vector in $V$ gives
the same distribution, and
$|\langle v,u\rangle|^2$ has the $\operatorname{Beta}(1,d-1)$ distribution.
Its upper tail is $(1-b^2)^{d-1}$, which proves the equality in
\eqref{eq:complex-cap}.

\smallskip
\noindent\emph{Normalized Frobenius radius.}
Fix $R<\sqrt2$ and put $a=1-R^2/2>0$.  The event
$\rho_N(I,U)\le R$ implies
\begin{equation}\label{eq:trace-tail}
 \sum_{j=1}^N\operatorname{Re}U_{jj}\ge aN.
\end{equation}
Set $b=a/2$ and $\gamma=(a-b)/(1-b)=a/(2-a)$.  Since every diagonal real
part is at most one, if only $q$ indices satisfied
$\operatorname{Re}U_{jj}\ge b$, then the sum in \eqref{eq:trace-tail} would
be at most
\[
 q+(N-q)b=bN+(1-b)q.
\]
For this to be at least $aN$, one must have $q\ge\gamma N$.  Thus
\eqref{eq:trace-tail} forces at least $\gamma N$ large diagonal entries.  Let
$m=\lfloor\gamma N\rfloor$.  For a fixed set $S$ of $m$ indices, expose the
corresponding columns successively.  Conditional on the preceding exposed
columns, the next column is uniform on their orthogonal complement, whose
complex dimension is $N-r+1$ at the $r$th exposure.  Applying
\eqref{eq:complex-cap} at every step gives
\begin{equation}\label{eq:fixed-subset-cap}
 \Pp\{\operatorname{Re}U_{jj}\ge b\ \forall j\in S\}
 \le(1-b^2)^{\sum_{r=1}^m(N-r)}.
\end{equation}
If at least $\gamma N$ diagonal entries are large, some set $S$ of size $m$
has the property in \eqref{eq:fixed-subset-cap}.  There are at most $2^N$
such sets.  A union bound therefore gives
\[
 H_N\{U:\rho_N(I,U)\le R\}
 \le2^N(1-b^2)^{mN-m(m+1)/2}
 \le e^{-c_RN^2+O_R(N)},
\]
Here $m=\gamma N+O(1)$, so the exponent
$mN-m(m+1)/2$ is a positive constant times $N^2$, up to an $O_R(N)$ error.
This proves \eqref{eq:normalized-small-ball}.

\smallskip
\noindent\emph{Fixed unnormalized Frobenius radius.}
For a fixed absolute radius $r>0$, the event
$\|I-U\|_{\HS}\le r$ implies
\[
 \sum_{j=1}^N(1-\operatorname{Re}U_{jj})\le r^2/2.
\]
At least $N/2$ indices consequently satisfy
$\operatorname{Re}U_{jj}\ge1-r^2/N$: otherwise more than half of the
nonnegative deficits $1-\operatorname{Re}U_{jj}$ would exceed $r^2/N$.
Take $m=\lfloor N/2\rfloor$ and
$b_N=1-r^2/N$ in \eqref{eq:fixed-subset-cap}.  For all large $N$,
$1-b_N^2\le2r^2/N$, and hence
\begin{align*}
 H_N\{U:\|I-U\|_{\HS}\le r\}
 &\le2^N\left(\frac{2r^2}{N}\right)^{mN-m(m+1)/2}\\
 &\le\exp\{-c_rN^2\log N+O_r(N^2)\}.
\end{align*}
The cap exponent is of order $N^2$, while
$\log(2r^2/N)=-\log N+O_r(1)$.  This supplies the additional factor
$\log N$ and proves \eqref{eq:absolute-small-ball}.

\section{Verification of the weighted contraction}\label{app:weighted-verification}

\begin{proof}[Proof of Lemma~\ref{lem:weighted-contraction}]
The proof has five steps.  We first write the discrete antisymmetric matrix,
then identify its continuum column sums and bound them.  We repeat the
column-sum comparison for the symmetric recursion and finally transfer both
continuum bounds back to the finite matrices.

\smallskip
\noindent\emph{Step 1: the compressed antisymmetric matrix.}
Equation \eqref{eq:anti-centered} maps every cosine mode to a combination of
even-frequency modes.  After this first application, the odd-frequency
coordinates are never needed.  Relabel $g_{2s}$ as mode $s$ and put
$L=N/4$.  The input and output indices of the compressed recursion then both
range over $1,\ldots,L$.  In these coordinates its signed coefficient matrix
is
\begin{equation}\label{eq:compressed-anti}
 S_L(s,r)=(1-r/L)^2\one_{s=\varphi_{2L}(2r)}
 -\frac{(2r-s)_++\one_{s<L}(2r-2L+s)_+}{2L^2},
\end{equation}
where $1\le r,s\le L$, $u_+=\max(u,0)$, and
\[
 \varphi_n(j)=\min\{j\bmod n,n-(j\bmod n)\}
\]
folds a frequency back to the first half of the circle.  For $s<L$, a folded
output can receive lower-frequency contributions from the two distinct
indices $s$ and $2L-s$.  At the endpoint $s=L$ these indices coincide, which
accounts for the indicator in \eqref{eq:compressed-anti}.

For the positive majorant we deliberately retain both terms even at that
endpoint and set
\[
 A_L^-(s,r)=(1-r/L)^2\one_{s=\varphi_{2L}(2r)}
 +\frac{(2r-s)_++(2r-2L+s)_+}{2L^2}.
\]
The lower-frequency coefficient in $A_L^-$ is at least the corresponding
coefficient in the exact signed matrix, so the triangle inequality gives
$|S_L(s,r)|\le A_L^-(s,r)$ entrywise.  This endpoint overcount is harmless
and lets us use the simpler majorant in the continuum calculation.  It is
therefore enough to bound the weighted column sums
\[
 \mathcal C_L^-(r)
 :=\frac{1}{w_L(r)}\sum_{s=1}^Lw_L(s)A_L^-(s,r),
 \qquad 1\le r\le L.
\]

\smallskip
\noindent\emph{Step 2: the continuum antisymmetric column sum.}
Set $x=r/L$.  The leading term of \eqref{eq:compressed-anti} lies at the
folded frequency $\varphi_{2L}(2r)$; after division by $L$, its location is
\[
 \tau(x)=2\min(x,1-x).
\]
The remaining coefficients form a Riemann sum whose limiting kernel is
\[
 K_x(y)=(2x-y)_++(2x-2+y)_+.
\]
Substitution of the weight \eqref{eq:weight} gives the following continuum
candidate; Step~5 will justify that
$\mathcal C_L^-(r)=R_0(r/L)+o(1)$ uniformly in $r$:
\begin{equation}\label{eq:R0}
 R_0(x)=
 (1-x)^2\sqrt{\frac{x}{\tau(x)}}
 \frac{1+3\tau(x)^2}{1+3x^2}
 +\frac{\sqrt x}{2(1+3x^2)}
 \int_0^1K_x(y)y^{-1/2}(1+3y^2)\,dy,
\end{equation}
with the endpoint values interpreted by continuity.  The first term in
$R_0$ is the contribution of the single folded leading mode; the integral
is the contribution of all lower-frequency modes.

\smallskip
\noindent\emph{Step 3: a uniform bound for $R_0$.}
For $0\le x\le1/2$, direct integration gives
\begin{equation}\label{eq:R0-low}
 R_0(x)=\frac{(1-x)^2}{\sqrt2}\frac{1+12x^2}{1+3x^2}
 +\frac{(4\sqrt2/3)x^2+(48\sqrt2/35)x^4}{1+3x^2}
 \le\frac45.
\end{equation}
For $1/2\le x\le1$, put $a=\sqrt{2-2x}$.  Clearing positive denominators
and using direct integration gives
\[
 R_0(x)=\frac{\sqrt{x}}{2(1+3x^2)}
 \left\{\frac{32}{5}(1-a^2)+\frac{11}{6}a^3
 +\frac{129}{70}a^7\right\}.
\]
Both sides of the desired inequality $R_0(x)\le4/5$ are nonnegative.
Substituting $x=1-a^2/2$, squaring, and clearing the remaining positive
denominators therefore preserves equivalence and reduces the claim to
$p(a)\ge0$ on $0\le a\le1$, where
\begin{align*}
 p(a)={}&149769a^{14}-299538a^{12}+297990a^{10}-1040256a^9
 -595980a^8\\
 &+3120768a^7+275233a^6-3115392a^5+493822a^4
 +3104640a^3\\
 &-3838464a^2-2069760a+3612672.
\end{align*}
Subsection~\ref{app:certificates} records positive Bernstein coefficients for
this polynomial and for the low-half check.  Those certificates are included
so that the numerical-looking bound is a finite exact verification rather
than a plot or floating-point computation.  Thus $\sup R_0\le4/5$.

\smallskip
\noindent\emph{Step 4: the symmetric limiting column sum.}
For the symmetric recursion, separate the folded leading term in
\eqref{eq:sym-centered}.  The lower-frequency Fej\'{e}r contribution is
supported on even modes, whereas the diagonal-pair correction contributes
to every mode.  After taking absolute values, an odd output therefore sees
only the diagonal contribution, whose limiting size is $x^2$; an even
output sees the difference between the two contributions, whose limiting
size is $|K_x(y)-x^2|$.  Repeating the weighted column calculation from
Step~2 gives the limiting upper bound
\begin{align}\label{eq:Rplus}
 R_+(x)={}&
 (1-x)^2\sqrt{\frac{x}{\tau(x)}}
 \frac{1+3\tau(x)^2}{1+3x^2}\notag\\
 &+\frac{\sqrt x}{2(1+3x^2)}
 \int_0^1\{x^2+|K_x(y)-x^2|\}
 y^{-1/2}(1+3y^2)\,dy.
\end{align}
Let $x_0=2-\sqrt2$ and $y_0=x(2-x)$.  For $x\ge x_0$, one has
$K_x(y)\ge x^2$ everywhere and hence $R_+(x)=R_0(x)$.  For $x<x_0$,
\begin{equation}\label{eq:Delta}
 R_+(x)\le R_0(x)+\Delta(x),\qquad
 \Delta(x)=\frac{x^{5/2}}{1+3x^2}
 \left\{\frac{16}{5}-2\sqrt{y_0}-\frac65y_0^{5/2}\right\}.
\end{equation}
For $x\le1/3$, $\Delta(x)\le4/(15\sqrt3)$.  On
$[1/3,x_0]$, bounding the increasing and decreasing factors at opposite
endpoints gives
\[
 \Delta(x)\le
 \frac{x_0^{5/2}}{1+3x_0^2}
 \left\{\frac{16}{5}-\frac{2\sqrt5}{3}
 -\frac65\left(\frac59\right)^{5/2}\right\}
 <0.186<\frac15.
\]
Therefore $\sup_xR_+(x)<1$.

\smallskip
\noindent\emph{Step 5: return to the finite matrices.}
It remains to pass from the limiting integrals back to the finite coefficient
matrices.  This is a uniform Riemann-sum estimate, including the integrable
singularity $y^{-1/2}$ at the origin.  For
$\alpha>-1$, uniformly in $0\le q\le L$,
\begin{equation}\label{eq:power-sum-uniform}
 \left|L^{-\alpha-1}\sum_{j=1}^qj^\alpha
 -\int_0^{q/L}y^\alpha\,dy\right|
 \le C_\alpha L^{-\min(1,\alpha+1)}.
\end{equation}
To obtain \eqref{eq:power-sum-uniform}, compare each summand with the
integral over its adjacent mesh interval.  Ordinary integral comparison
applies for $\alpha\ge0$; for $-1<\alpha<0$, begin at the first mesh point
to isolate the integrable singularity.  The resulting error is uniform in
the terminal index $q$.  Restricting the sum to one parity class doubles the
mesh size and gives half the same limiting integral, with the same order of
error.

The folded kernels have only finitely many breakpoints.  Between successive
breakpoints, every discrete weighted sum above is a linear combination of
\eqref{eq:power-sum-uniform} with
$\alpha=-1/2,1/2,3/2,5/2$.  Hence the antisymmetric and symmetric column sums
converge to $R_0$ and $R_+$ uniformly over all columns.  Choose $\rho_-$
strictly between $\sup_xR_0(x)$ and one, and choose $\rho_+$ strictly between
$\sup_xR_+(x)$ and one.  For all sufficiently large $L$, every finite
weighted column sum is then bounded by the corresponding $\rho_\pm$.  This
proves both assertions of the lemma.
\end{proof}

\subsection{Exact polynomial certificates}\label{app:certificates}

We record the exact positivity checks used in Step~3 above.  Recall that the
degree-$n$ Bernstein
basis on $[0,1]$ is
\[
 B_{j,n}(y)=\binom nj y^j(1-y)^{n-j},
 \qquad 0\le j\le n.
\]
Every $B_{j,n}$ is nonnegative on $[0,1]$.  Consequently, if a polynomial is
written as $\sum_j b_jB_{j,n}$ with all $b_j>0$, then the polynomial is
strictly positive throughout the interval.  The coefficient lists below are
therefore exact certificates for the two inequalities used in the proof.

For \eqref{eq:R0-low}, use $1/\sqrt2<3/4$ and $\sqrt2<3/2$, and put
$y=2x\in[0,1]$.  The difference between $4/5$ and the resulting rational
upper bound is
\[
 \frac{q(y)}{35(3y^2+4)},\qquad
 q(y)=-\frac{387}{4}y^4+315y^3-\frac{1309}{4}y^2+105y+7.
\]
In the degree-four Bernstein basis on $[0,1]$, the coefficients of $q$ are
\[
 7,\quad \frac{133}{4},\quad\frac{119}{24},\quad\frac78,\quad3.
\]
They are all positive, and the denominator $35(3y^2+4)$ is positive, so this
certifies the bound in \eqref{eq:R0-low}.  For the degree-fourteen polynomial
$p$ arising on the high half of the interval, the
Bernstein coefficients are
\begin{align*}
&3612672,\ 3464832,\ \frac{42572544}{13},\ \frac{39664800}{13},
\ \frac{400805650}{143},\ \frac{362914778}{143},\\
&\frac{88204643}{39},\ \frac{847984627}{429},\ \frac{722244916}{429},
\ \frac{197950372}{143},\ \frac{155244460}{143},\\
&\frac{10426436}{13},\ \frac{50185120}{91},\ \frac{2261328}{7},\ 95504.
\end{align*}
These coefficients are again strictly positive, proving $p(a)>0$ on
$[0,1]$ and completing the high-half check.

\section*{Declarations}

\paragraph{Data and code availability.}
No datasets were generated or analyzed in this study.

\paragraph{Use of AI-powered tools.}
AI-powered tools were used in the preparation of this manuscript for
literature search, computational checking, exploring proof ideas, and
editorial assistance.  The author takes full responsibility for the
mathematical arguments, accuracy, and final text.

% EJP asks authors to include the generated bibliography directly in the
% manuscript source.  The entries below were generated with amsplain.
\providecommand{\bysame}{\leavevmode\hbox to3em{\hrulefill}\thinspace}
\providecommand{\MR}{\relax\ifhmode\unskip\space\fi MR }
\providecommand{\MRhref}[2]{%
  \href{http://www.ams.org/mathscinet-getitem?mr=#1}{#2}}
\providecommand{\href}[2]{#2}

\begin{acks}
The author is grateful to Persi Diaconis for introducing him to the problems
related to the quantum baker walk.
\end{acks}

\end{document}